\documentclass[12pt]{article}
\pdfoutput=1
\usepackage{graphics,epsfig,amssymb}
\usepackage{amsmath}
\usepackage{amsthm}
\usepackage{amsfonts}
\usepackage{amssymb,latexsym}
\usepackage[all]{xy}
\usepackage{txfonts}
\usepackage{amssymb}
\usepackage{graphicx}
\usepackage{amsfonts}
\usepackage{verbatim}
\usepackage{mathrsfs}
\usepackage{color}
\usepackage{hyperref}

\usepackage{tikz}
\usetikzlibrary{calc,matrix}
\theoremstyle{plain}
    \newtheorem{thm}{Theorem}[section]
    \newtheorem{prop}{Proposition}[section]
    \newtheorem{lemma}{Lemma}[section]
    \newtheorem{cor}{Corollary}[section]
    \newtheorem{defn}{Definition}[section]

    \newtheorem{example}{Example}[section]

\numberwithin{equation}{section}
\usepackage{graphicx}
\usepackage{graphicx}
\usepackage{pict2e}
\usepackage{slashbox}
\usepackage{arcs}
 \usepackage{graphicx,amssymb,amstext,amsmath}

\begin{document}

\title{Separable functions:  symmetry, monotonicity,  and  applications\thanks{Research was supported partially by the
National Natural Science Foundation of P. R. China (Grant No.
11571371); the Natural Science Foundation of Hunan Province (Grant No.
2018JJ3136); Scientific Research Fund of Hunan University of Science and Technology
(No. E51794).}}
\author{Tao Wang,
Taishan Yi}

\date {}
\maketitle

\begin{abstract}
In this paper,  we  introduce  concepts of separable functions in balls and in the whole space, and develop a new method to investigate the qualitative  properties of separable functions. We first study  the  axial symmetry and monotonicity of separable functions in unit circles by geometry analysis, and we prove the uniqueness of the
 symmetry axis for  nontrivial separable functions. Then by
  using reduction dimension  and convex analysis, we get the axial symmetry and monotonicity of separable functions in high dimensional spheres. Based on the above results on unit circles and spheres, we deduce the axial symmetry and monotonicity of separable functions in balls and
  the radial symmetry and monotonicity of separable functions in the whole space.  Conversely,  the function with axial symmetry and monotonicity in the ball domain is separable function, and
   the function with radial symmetry and monotonicity in the whole space  is also separable function. These enable us to provide easily some examples
 that separable functions in balls may be just axially symmetric not radially symmetric. Finally, as applications, we  obtain the axial  symmetry and monotonicity of all the  positive ground states to the   Choquard equation in a ball as well as  the radial  symmetry and monotonicity of all the  positive ground states   in the whole space.
  \end{abstract}
\bigskip



\section{Introduction}
As we know,
symmetry and monotonicity are very important properties of solutions to elliptic partial differential equations, see \cite{cp,fl,kmk}. They  play an essential role in the uniqueness and dependence on parameters of solutions and hence have been extensively  investigated, see \cite{CHENLI,clo,MZ,wtyt} and references therein. In  these literatures, the maximum principle is vital  to study various  properties of solutions.
Based on maximum principle,  a solution $u$ to elliptic equations is comparable with the mirror point about  a hyperplane. To be precise, let $H\subset \mathbb{R}^n$ be an open half-space. For $x\in \mathbb{R}^n$, $\sigma_Hx$ denotes the symmetric point of $x$ with respect to the hyperplane $\partial H$. Then $u$ keeps the larger value  on $H$ (or on $\mathbb{R}^N\setminus H$).
 It is natural to guess that sufficiently many such hyperplanes can lead to symmetry and monotonicity.
This motivates us to introduce the following concept of separable functions.
\begin{defn}
 Let $\Omega\subset \mathbb{R}^N$. A function $u:\Omega\to\mathbb{R}$ is said to be \emph{separable} in $\Omega$ if for any $H\in \mathcal{H}_\Omega:=\{\mbox{open half space }H\subset \mathbb{R}^N| \sigma_H(H\bigcap \Omega)=\Omega\backslash cl (H)\}$,
\begin{equation}\label{47}
\mbox{ either $u(x)\geq u(\sigma_Hx)$ for all $x\in H\bigcap \Omega $ or
$u(x)\leq u(\sigma_Hx)$ for all $x\in H\bigcap \Omega$.}
\end{equation}
\end{defn}

In view of  the above definition, we see that $\Omega$ has better symmetry if the set $\mathcal{H}_\Omega$ is larger, and then the symmetry  and monotonicity of  separable functions in $\Omega$ can be simpler and richer. Since spheres, balls, and the whole space have rich symmetry, in the present paper, we mainly investigate symmetry and monotonicity  of separable functions in these domains.

For motivations of this study, first we recall some tools used to study symmetry and monotonicity of solutions to elliptic equations, which include  symmetric decreasing rearrangement (or Schwarz symmetrization), polarization method, the method of moving planes and its variants. The symmetric decreasing rearrangement mainly depends on   rearrangement inequalities and minimizing method to obtain the existence and symmetry of the minimizer (see \cite{LEH,llm}). For the polarization  method, one can first establish polarization inequality and then show the relationship between the solution and its polarization, via which the symmetry can be proved, (see \cite{btw,brock,sjvm,schaftingen} and references therein). As one of powerful tools in establishing symmetry and monotonicity of  solutions to elliptic equations, the method of moving planes was proposed by the Soviet mathematician Alexanderoff in the early 1950s. Decades later, it was further developed by  Gidas, Ni and Nirenberg \cite{GNN}, Chen and Li\cite{CHENLI} and many others. Please see
\cite{CHENLI1,Li2,SZOU,Li4,Linchang,clo1,gmx} and references therein. It is known that the three tools essentially rely on the  elliptic equations and the specific solutions. Then a natural question is whether  we can study the symmetry and monotonicity of the solutions by just using their separability instead  of the elliptic equations and other properties of the solutions. In other words, whether we can study the symmetry and monotonicity only via separability. This paper will give an affirmative answer.

In this paper,  we successively  consider  separable functions in circles, spheres, balls, and
the whole space. We leave complicated domains for future study. Here we  sketch the main ideas and approaches  to study several separable functions. To be precise,
we first introduce the concepts of separable functions in unit circle. Then by employing geometric analysis, we obtain that the set of global extremal points for a given positive and nonconstant separable function in unit circle is two arcs. One arc (max-arc, for short) is the set of maximum points and the other arc (min-arc, for short) is the set of minimum point of this function. This, combined with the separability  of the function, implies that the centers of the max-arc and min-arc are the ends of the same diameter.  By choosing suitable diameter and using the separability again, we prove the axial symmetry and monotonicity of separable functions  in circles with the unique symmetry axis.

 In what follows, in order to apply reduction dimension method, we give  equivalent definitions of separable functions in high dimensional spheres. For a given  positive and nonconstant separable function  in a sphere,  we have shown  that the set of global extremal points of this function is two sphere caps by using reduction dimension method and convex analysis. One sphere cap (max-cap, for short) is the set of maximum points and the other sphere cap  (min-cap, for short) is the set of minimum point of this function. Based on the separability of the function constrained in the unit circle through the centers of the max-cap and min-cap, we show that the centers  of the max-cap and min-cap are in the same diameter. By constructing suitable circles and using the axial symmetry and monotonicity of  this function in these circles, it is easy to check that this separable function in a given sphere is axially symmetric and monotone with  the unique symmetry axis.

In the sequel, based on the fact that a ball is made up of homocentric spheres, by using the separability of functions in the ball, we point out that the centers of the max-caps and min-caps for all the  homocentric spheres are in the same diameter. So we can deduce the axial symmetry and monotonicity of separable functions in balls by applying axial symmetry and monotonicity of separable functions in spheres. Conversely, the axially symmetric and monotone functions in balls are also  separable. In other words, the separability is  equivalent to  the axial symmetry and monotonicity for a given function in balls. This observation  enables us  to give an example that   separable functions  in balls  may be only axially symmetric  but not radially symmetric.

Finally, we give the definitions of separable functions in the whole space. Note that a positive separable function  in the  the whole space is separable  in any ball. This fact, combined with the axial symmetry and monotonicity of the  separable functions in balls,  implies that the separable function  in the whole space admits an unique symmetry axis  passing through any given point, and all the symmetry axes are parallel to each other. Furthermore,  suppose that the infimum of the separable function in the whole space is zero. Then we can deduce the radial symmetry and monotonicity of the separable function in the whole space.  Similarly, we also easily see that the separability is  equivalent to  the  radial symmetry and monotonicity for a given  positive function with the infimum  being zero.

As applications of symmetry and monotonicity of  separable functions, we consider Choquard equations in balls and in the whole space. Specifically, we  obtain the axial  symmetry and monotonicity of all the  positive ground states to  Choquard equations in balls as well as  the radial  symmetry and monotonicity  of all the  positive ground states  to  Choquard  equations in the whole space.

To sum up, this paper  provides  a new perspective to study the symmetry and monotonicity of solutions to elliptic equations. Roughly speaking, it involves  two steps to obtain the symmetry and monotonicity of solutions. In first step, we  prove the symmetry and monotonicity of separable functions.  We emphasize that the proof of this step does not rely on the exact equations or  properties of specific solutions.
 In second step, we are concerned with the separability of a specific  solution to a concrete  equation, and then deduce  its symmetry and monotonicity.

 Throughout this paper, we always  assume  that separable functions are only continuous.

For convenience, we  introduce  some  notations as follows:

\noindent $\bullet$ $\mathbb{N}$ is the set of all the positive integers.

\noindent $\bullet$ $N\in \mathbb{N}$ and $ N\geq2$.

\noindent $\bullet$ $0_k:=\underbrace{(0,\cdots,0) }\limits_{k}$ for $k\in \mathbb{N}$.

\noindent $\bullet$ Let
$$\mathbf{aff}(A):=\{\sum^{m}_{i=1}\lambda_ix^i|m\in \mathbb{N}, x^i\in A, \lambda_i\in \mathbb{R} \mbox{ and } \sum^{m}_{i=1}\lambda_i=1\}$$
and
$$\mathbf{co}(A):=\{\sum^{l}_{i=1}\lambda_ix^i|l\in \mathbb{N}, x^i\in A, \lambda_i\in [0,1] \mbox{ and } \sum^{l}_{i=1}\lambda_i=1\}.$$

\noindent$\bullet$ $S^{N-1}(x)$  is the  unit sphere centered at  $x$ and $S_r^{N-1}(x)$  is the sphere centered at  $x$  with radius $r>0$ in  $\mathbb{R}^N$.
For simplicity of notations, we write $S^{N-1}(0)$ and
$S_r^{N-1}(0)$
as $S^{N-1}$ and $S_r^{N-1}$, respectively.

\noindent$\bullet$
For  $r>0$, $B_r(x)$ is the closed ball  centered at $x\in\mathbb{R}^N$ with radius $r$. For simplicity of notations, we write   $B_r(0)$  as $B_r$.

\noindent$\bullet$ $O(N)$ represents the set of orthogonal transformations in $\mathbb{R}^N$.
For $M\in O(N)$ and $u\in C(\mathbb{R}^N,\mathbb{R})$, we define $u_{M}(x):=u(M^{-1}x)$.

\noindent$\bullet$ Let $H\subset \mathbb{R}^N$ be an open half-space. For any $x\in \mathbb{R}^N$,  $\sigma_Hx$ is the symmetric  point of $x$ with respect to $\partial H$.

The remaining of this  paper is organised as follows. In section 2,
 the  axial  symmetry and monotonicity of   separable  functions  in   circles,   spheres, and balls are established, by using  reduction dimension method, geometric analysis, and convex analysis.  Based on these results,  section 3 is devoted to the proof of radial  symmetry and monotonicity of   separable  functions in the whole space. Finally, in section 4, we apply our main theoretical results to Choquard type equations to obtain the axial  symmetry and monotonicity of all the  positive ground states  in a ball as well as  the radial  symmetry and monotonicity of all the  positive ground states   in the whole space.

\section{ Separable functions   in bounded domains}
In this section, we  investigate  the   symmetry and monotonicity of   separable  functions  in bounded domains such as high dimensional  balls $B_R\subset \mathbb{R}^N$ by using dimensionality reduction, geometry analysis, and convex analysis.

In the following, we first give some notations. Let $S^1$ be the unit circle in $\mathbb{R}^2$. Let $\mathop{xy}\limits^{\curvearrowright}$ be an arc from  $x$ to $y$  counterclockwise in $S^1$ and  $s(\mathop{xy}\limits^{\curvearrowright})$ be the arc length of $\mathop{xy}\limits^{\curvearrowright}$. Let
$$S_{\alpha}^1=\{(\cos(\alpha+\theta),\sin(\alpha+\theta)):\theta\in(0,\pi)\}, \ B_{\alpha}^1=\bigcup_{r\in[0,1]}rS_{\alpha}^1,\   l_{\alpha}=\partial B_{\alpha}^1\backslash S_{\alpha}^1,$$ where $\alpha\in\mathbb{R}.$ For any  $x\in S^1,$  let $l_{\alpha}(x)$ be the axial symmetric  point of $x$  with respect to $l_{\alpha}$.

For a given line $L\subseteq \mathbb{R}^N$, we say that $x,y\in \mathbb{R}^N$ are  axially symmetric  with respect to $L$ if there is $z^*\in L$ such that $||x-z^*||=||y-z^*||$ and $x-z^*,z^*-y\bot L$, respectively. Here $||x-z^*||=\min\{||x-z||:z\in L\}, ||y-z^*||=\min\{||y-z||:z\in L\}$.
A function $u\in C(\mathbb{R}^N,\mathbb{R})$ is said to be axially symmetric with respect to  a line $L$ if $u(x)=u(y)$
for any   $x,y\in \mathbb{R}^N$ that are axially symmetric with respect to $L$. Here $L$ is a symmetry axis of $u$.

Now we begin with the definition and properties of  separable functions   in  $S^1$ in the following subsection.
\subsection{ Separable functions   in unit circles}
The following gives the definition of  separable functions in unit circle $S^1$.

\begin{defn}
A function $v\in C(S^1,\mathbb{R})$ is said to be separable in $S^1$, if
for any  $\alpha\in[0,2\pi),$  there holds
\begin{equation}\label{1}
\mbox{either $v(l_{\alpha}(x))\geq v(x)$ for all $x\in S_{\alpha}^1$ or
$v(l_{\alpha}(x))\leq v(x)$ for all $x\in S_{\alpha}^1$}.
\end{equation}
\end{defn}
Now we show that the properties of separable functions   in $S^1$, which plays a critical role in investigating the symmetry and monotonicity of separable functions in high dimensional spheres and balls.
\begin{lemma}\label{lem2.1}
Let $v\in C(S^1,(0,\infty))$ be a separable function in $S^1$.
Suppose that  $\max\limits_{S^1}v>\min\limits_{S^1}v.$ Then there exist $\alpha_0\in[0,2\pi)$ and $\theta_1,\theta_2\in[0,\pi)$ such that
\begin {itemize}
\item[\rm{(i)}] $\theta_1+\theta_2<\pi.$

\item[\rm{(ii)}] $v^{-1}(\max\limits_{S^1} v)=\{(\cos(\alpha_0+\theta),\sin(\alpha_0+\theta)): |\theta|\leq \theta_1\}$ and $$v^{-1}(\min_{S^1} v)=\{(\cos(\alpha_0+\pi+\theta),\sin(\alpha_0+\pi+\theta)): |\theta|\leq \theta_2\}.$$

    \item[\rm{(iii)}]  $v(x)=v(l_{\alpha_0}(x))$ for all  $x\in S_{\alpha_0}^1$.

    \item[\rm{(iv)}] $v((\cos \alpha, \sin \alpha))$ is not a  constant function and is a nonincreasing function with respect to $\alpha\in[\alpha_0,\alpha_0+\pi)$.
\end {itemize}
\end{lemma}
\begin{proof}
Let $A:=v^{-1}(\max\limits_{S^1}v)$ and $B:=v^{-1}(\min\limits_{S^1}v)$. Then $A\neq \emptyset$ and  $B\neq \emptyset$.
 We shall finish the proof by the following five steps.

 \noindent\textbf{Step 1.} We  claim that there exist $x\in A$ and $y\in B$ such that $||x-y||=2$, that is, there exists $\alpha_0\in [0,2\pi)$
 such that $x=(\cos \alpha_0, \sin \alpha_0)\in A$ and $y=(\cos (\alpha_0+\pi), \sin (\alpha_0+\pi))\in B,$ where  $||\cdot||$ represents the Euclidean norm on $\mathbb{R}^2$.

Otherwise,  according to the compactness of $\mbox{ $A$ and $B$} $, there exist $x\in A, y\in B$ such that
$||x-y||=\max d(A\times B)<2$  where $d:A\times B\to\mathbb{R}$ by $(x,y)\mapsto ||x-y||$.
Without loss of generality, we assume $0<s(\mathop{xy}\limits^{\curvearrowright})<\pi.$ Clearly, there exists $\alpha_0\in[0,2\pi)$ such that
 $$x=(\cos \alpha_0,\sin \alpha_0),\ y=(\cos (\alpha_0+s(\mathop{xy}\limits^{\curvearrowright})),
\sin (\alpha_0+s(\mathop{xy}\limits^{\curvearrowright}))).$$
By taking  $ {z}=(\cos(\alpha_0+2s(\mathop{xy}\limits^{\curvearrowright})),
\sin(\alpha_0+2s(\mathop{xy}\limits^{\curvearrowright})))$, we have $s(\mathop{xy}\limits^{\curvearrowright})=s(\mathop{yz}\limits^{\curvearrowright})\in(0,\pi)$, and hence $\mathop{{ {z}x}}\limits^{\curvearrowright}\bigcap A\subset \{x, {z}\}.$
Let $\alpha^*=\alpha_0-\frac{2(\pi-s(\mathop{xy}\limits^{\curvearrowright}))}{3}.$
Then $0<\alpha_0-\alpha^*<\pi-s(\mathop{xy}\limits^{\curvearrowright})$, and thus $l_{\alpha^*}(x)\in \mathop{zx}\limits^{\curvearrowright},$
$s(
\begin{tikzpicture}[>=stealth,baseline,anchor=base,inner sep=0pt]
\matrix (foil) [matrix of math nodes,nodes={minimum height=0.5em}] {
x& l_{\alpha^*}&  (&y&)\\
};
\path[->] ($(foil-1-1.north)+(-.5ex,.5ex)$)   edge[bend left=30]    ($(foil-1-5.north)+(0ex,.5ex)$);
\end{tikzpicture}
)=\frac{2\pi+s(\mathop{xy}\limits^{\curvearrowright})}{3}\in(s(\mathop{xy}\limits^{\curvearrowright}),\pi),$
and $x,y\in S_{\alpha^*}^1.$ So, $v(l_{\alpha^*}(x))<\max\limits_{S^1}v$  due to $l_{\alpha^*}(x)\in \mathop{zx}\limits^{\curvearrowright}$ and the choices of $x,y$.  It follows from \eqref{1} and $v(l_{\alpha^*}(x))<\max\limits_{S^1}v=v(x)$
that
$$ v(\bar{x})\geq v({l_{\alpha^*}}(\bar{x}))\ \mbox{for all }\bar{x}\in S_{\alpha^*}^1.$$
In particular, $v(y)\geq v(l_{\alpha^*} (y))$  and hence   $l_{\alpha^*}(y)\in B$,
 a contradiction with $s(\begin{tikzpicture}[>=stealth,baseline,anchor=base,inner sep=0pt]
\matrix (foil) [matrix of math nodes,nodes={minimum height=0.5em}] {
x& l_{\alpha^*}&  (&y&)\\
};
\path[->] ($(foil-1-1.north)+(-.5ex,.5ex)$)   edge[bend left=30]    ($(foil-1-5.north)+(0ex,.5ex)$);
\end{tikzpicture})\in(s(\mathop{xy}\limits^{\curvearrowright}),\pi).$
The above arguments are illustrated in Figure~\ref{figure1.jpg}. Therefore, we have finished the proof of Step 1.
\begin{figure}[!ht]
  \center
  \includegraphics[height=5cm,width=6cm]{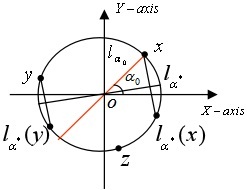}
  \caption{schematic diagram  for the proof of \textbf{Step 1}}\label{figure1.jpg}
 \end{figure}

\noindent\textbf{Step 2.}  We shall prove $A=\mathop{{\bar{\bar{x}}\bar{x}}}\limits^{\curvearrowright}$, where $\bar{x}\in A\bigcap \mathop{xy}\limits^{\curvearrowright},\bar{\bar{x}}\in A\bigcap \mathop{yx}\limits^{\curvearrowright}$ with
$\|\bar{x}-y\|=\min\limits_{z\in A\bigcap \mathop{xy}\limits^{\curvearrowright}}\|z-y\|$ and $\|\bar{\bar{x}}-y\|=\min\limits_{z\in A\bigcap \mathop{yx}\limits^{\curvearrowright}}\|z-y\|$.

Clearly, $A\subset \mathop{{\bar{\bar{x}}\bar{x}}}\limits^{\curvearrowright}.$  We only need to prove
$\mathop{{\bar{\bar{x}}\bar{x}}}\limits^{\curvearrowright}\subset A.$  It is clear that  $\mathop{{\bar{\bar{x}}\bar{x}}}\limits^{\curvearrowright}\subset A$ if $s(\mathop{{\bar{\bar{x}}\bar{x}}}\limits^{\curvearrowright})=0$. Now we suppose $s(\mathop{{\bar{\bar{x}}\bar{x}}}\limits^{\curvearrowright})>0$.
Take $\bar{x}^*,\bar{\bar{x}}^*\in \mathop{{\bar{\bar{x}}\bar{x}}}\limits^{\curvearrowright}$ such  that
$$\begin{array}{lll}
s(\mathop{{\bar{x}^*\bar{x}}}\limits^{\curvearrowright})=
\sup\{s(\mathop{{\bar{z}\bar{x}}}\limits^{\curvearrowright}):\mathop{{\bar{z}\bar{x}}}\limits^{\curvearrowright}\subset A \},
s(\mathop{{\bar{\bar{x}}\bar{\bar{x}}^*}}\limits^{\curvearrowright})=
\sup\{s(\mathop{{\bar{\bar{x}}\bar{z}}}\limits^{\curvearrowright}):\mathop{{\bar{\bar{x}}\bar{z}}}\limits^{\curvearrowright}\subset A\},
\end{array}.$$
It suffices to prove $\bar{\bar{x}}^*=\bar{x}^*.$  Otherwise, $\bar{\bar{x}}^*\neq\bar{x}^*.$
Without loss of generality, we may assume that
$s(\mathop{{\bar{x}^*\bar{x}}}\limits^{\curvearrowright})\geq s(\mathop{{\bar{\bar{x}}\bar{\bar{x}}^*}}\limits^{\curvearrowright}).$
Take $\bar{x}^0\in \mathop{{\bar{x}^*\bar{x}}}\limits^{\curvearrowright}$ and $\alpha^*\in[0,2\pi)$
 such that $s(\mathop{{\bar{x}^0\bar{x}}}\limits^{\curvearrowright})=
 s(\mathop{{\bar{\bar{x}}\bar{\bar{x}}^*}}\limits^{\curvearrowright})$ and $\bar{\bar{x}}=l_{\alpha^*}(\bar{x})$.
 Then $\min\{s(\mathop{{\bar{x}\bar{\bar{x}}}}\limits^{\curvearrowright}),s(\mathop{{\bar{\bar{x}}^*\bar{x}^0}}\limits^{\curvearrowright})\}>0$, $\mathop{{\bar{\bar{x}}\bar{\bar{x}}^*}}\limits^{\curvearrowright}\bigcup
 \mathop{{\bar{x}^0\bar{x}}}\limits^{\curvearrowright}\subset A,$ $\mathop{{\bar{x}^0\bar{x}}}\limits^{\curvearrowright}\subset S_{\alpha^*}^1$,
 and $\mathop{{\bar{\bar{x}}\bar{\bar{x}}^*}}\limits^{\curvearrowright}\subset S^1\backslash S_{\alpha^*}^1.$
For any $\alpha\in [\alpha^*,\alpha^*+\frac{\min\{s(\mathop{{\bar{x}\bar{\bar{x}}}}\limits^{\curvearrowright}),s(\mathop{{\bar{\bar{x}}^*\bar{x}^0}}\limits^{\curvearrowright})\}}{2}]$, we easily check that $\mathop{{\bar{x}^0\bar{x}}}\limits^{\curvearrowright}\subset S_{\alpha}^1$
,$\mathop{{\bar{\bar{x}}\bar{\bar{x}}^*}}\limits^{\curvearrowright}\subset S^1\backslash S_{\alpha}^1,$
$l_\alpha(\bar{\bar{x}})\in \mathop{{\bar{x}\bar{\bar{x}}}}\limits^{\curvearrowright}$, and thus
$v(\bar{\bar{x}})>v(l_\alpha(\bar{\bar{x}})).$ It follows from \eqref{1} that
$v(l_\alpha(\bar{x}^0))\geq v(\bar{x}^0)=\max\limits_{S^1} v,$ (see Figure~\ref{figure2.jpg}).
\begin{figure}[!ht]
  \center
  \includegraphics[height=5cm,width=6cm]{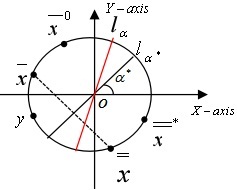}
  \caption{schematic diagram  for the partial  proof of \textbf{Step 2}}\label{figure2.jpg}
\end{figure}
As a result, $l_\alpha(\bar{x}^0)\in A$ for all $\alpha\in (\alpha^*,\alpha^*+\frac{\min\{s(\mathop{{\bar{x}\bar{\bar{x}}}}\limits^{\curvearrowright}),s(\mathop{{\bar{\bar{x}}^*\bar{x}^0}}\limits^{\curvearrowright})\}}{2})$ and hence by the definition of $l_\alpha$ and the compactness
of $A$,
$\mathop{{\bar{\bar{x}}\bar{\bar{x}}^*}}\limits^{\curvearrowright}\subsetneqq \begin{tikzpicture}[>=stealth,baseline,anchor=base,inner sep=0pt]
\matrix (foil) [matrix of math nodes,nodes={minimum height=0.5em}] {
\bar{\bar{x}}&l_{\alpha^*+\frac{\min\{s(\mathop{{\bar{x}\bar{\bar{x}}}}\limits^{\curvearrowright}),s(\mathop{{\bar{\bar{x}}^*\bar{x}^0}}\limits^{\curvearrowright})\}}{2}} &  (&\bar{x}^0&)\\
};
\path[->] ($(foil-1-1.north)+(-.5ex,.5ex)$)   edge[bend left=30]    ($(foil-1-5.north)+(0ex,.5ex)$);
\end{tikzpicture}\subset A,$ a contradiction with the choice of $\bar{\bar{x}}^*.$
This proves $\bar{x}^*=\bar{\bar{x}}^*$ and consequently $A=\mathop{{\bar{\bar{x}}\bar{x}}}\limits^{\curvearrowright}.$

\noindent\textbf{Step 3.} Show that there exist $\alpha_0\in [0,2\pi)$ and $\theta_1,\theta_2\in[0,\pi)$ such that
$A=\{(\cos(\alpha_0+\theta),\sin(\alpha_0+\theta)): |\theta|\leq \theta_1\}$ and $B=\{(\cos(\alpha_0+\pi+\theta),\sin(\alpha_0+\pi+\theta)): |\theta|\leq \theta_2\}.$

By \textbf{Step 2}, there exist $\alpha_0\in [0,2\pi)$ and $\theta_1\in[0,\pi)$ such that
$A=\{(\cos(\alpha_0+\theta),\sin(\alpha_0+\theta)): |\theta|\leq \theta_1\}$. By applying the claim in
\textbf{Step 2} again to $\tilde{v}(x):=1+\max\limits_{S^1}v-v(x)$, we have
$$\tilde{v}^{-1}(\max\limits_{S^1} \tilde{v})=\{(\cos(\alpha_1+\theta),\sin(\alpha_1+\theta)): |\theta|\leq \theta_2\}$$
for some $(\alpha_1,\theta_2)\in[0,2\pi)\times [0,\pi).$ In other words, $B=\{(\cos(\alpha_1+\theta),\sin(\alpha_1+\theta)): |\theta|\leq \theta_2\}.$ It suffices to prove that $(0,0)$ belongs to the line segment $\overline{x^*y^*}$, where
$x^*:=(\cos{\alpha_0},\sin{\alpha_0}), y^*:=(\cos{\alpha_1},\sin{\alpha_1})$. Otherwise, there exists a diameter $l$ such that
$x^*,y^*$ are on the same side of $l$ and $l\bigcap \{x^*,y^*\}=\emptyset$. By the choices of $x^*$ and $y^*$,
 there exist $\tilde{x}^*,\tilde{y}^*\in \mathop{x^*y^*}\limits^{\curvearrowright}$ such that $v(\tilde{x}^*)>v(l(\tilde{x}^*))$ and $v(\tilde{y}^*)<v(l(\tilde{y}^*))$.
 This, combined with the  separability of $v$,
 implies a contradiction with the fact that $v|_{S^1}$ is not constant.

\noindent\textbf{Step 4.}  We show that  $v(l^*(x))=v(x)$ for any $x\in \mathop{x^*y^*}\limits^{\curvearrowright},$ where $l^*=\overline{{x^*y^*}}$, and $x^*, y^*$  defined in Step 3 represent the centers of $A$ and $B$, respectively.

Otherwise, there exists $\bar{x}\in \mathop{x^*y^*}\limits^{\curvearrowright}\backslash(A\bigcup B)$ and
$v(l^*(\bar{x}))\neq v(\bar{x}).$
Without loss of generality, we may assume that $v(\bar{x})>v(l^*(\bar{x}))$ and $l^*(\bar{x})\in \mathop{y^*x^*}\limits^{\curvearrowright}\backslash(A\bigcup B).$
In view of the continuity of $v$ and the compactness of $A$, we know that there exists  $\bar{\bar{x}}\in \mathop{l^*(\bar{x})x^*}\limits^{\curvearrowright} \backslash A$
such that $v(\bar{x})>v(\bar{\bar{x}})$, $\{x^*,x^{**},\bar{\bar{x}}\}$ and $\{y^*,\bar{x}\}$  locate on both sides of the line $\bar{l}$,  with $x^{**}=(\cos(\alpha_0-\theta_1),\sin(\alpha_0-\theta_1))$ and $\bar{l}$ being the
perpendicular bisector of $\overline{\bar{x}\bar{\bar{x}}}$.  It follows from $\bar{l}(\bar{x})=\bar{\bar{x}}$, $v(\bar{x})>v(\bar{\bar{x}})$
and the separability of $v$ that $v(x^{**})\leq v(\bar{l}(x^{**}))$, where $x^{**}=(\cos(\alpha_0-\theta_1),\sin(\alpha_0-\theta_1)),$
which yields a contradiction to $\bar{l}(x^{**})\notin A$  (see Figure~\ref{figure3.jpg}).
\begin{figure}[!ht]
  \center
  \includegraphics[height=5cm,width=6cm]{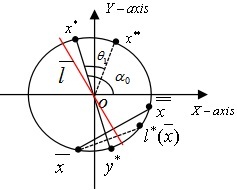}
  \caption{schematic diagram  for the   proof of \textbf{Step 4}}\label{figure3.jpg}
\end{figure}

\noindent\textbf{Step 5.}
We show that $u:[0,\pi]\ni \theta\mapsto v(\cos(\alpha_0+\theta),\sin(\alpha_0+\theta))\in(0,\infty)$ is decreasing at $\theta\in [0,\pi].$\\
Indeed, for any given $\theta_1^*,\theta_2^*\in [0,\pi]$  with $\theta_1^*<\theta_2^*$, let
$$\bar{x}=(\cos(\alpha_0+\theta_1^*),\sin(\alpha_0+\theta_1^*)),\bar{\bar{x}}=(\cos(\alpha_0+\theta_2^*),\sin(\alpha_0+\theta_2^*))$$
and $\bar{l}$ represent the perpendicular bisector of $\overline{\bar{x}\bar{\bar{x}}},$ that is, $\bar{l}(\bar{x})=\bar{\bar{x}}.$ Then
 $\{x^*,\bar{x}\}$  and  $\{y^*,\bar{\bar{x}}\}$   locate on both sides of the line  $\bar{l}.$
 It follows from \eqref{1} that
 $$v(\bar{x})\geq v(\bar{l}(\bar{x}))=v(\bar{\bar{x}}).$$
In other words, $u(\theta_1^*)\geq u(\theta_2^*).$ The arbitrariness of $\theta_1^*$ and $\theta_2^*$ implies that $u$
is decreasing.

Therefore, \textbf{Step 3} gives (i) and (ii) while (iii) and (iv) follow from \textbf{Step 4} and  \textbf{Step 5},
respectively.
\end{proof}

By Lemma \ref{lem2.1}, it is easily to  check  the following two corollaries, which are very useful in extending the conclusions in Lemma \ref{lem2.1} to separable functions in  high dimensional spheres and balls.

\begin{cor}\label{cor2.1}
Let $v$ and $\alpha_0$ be choose in Lemma~\ref{lem2.1}.  For any $\alpha\in \mathbb{R},$ if $(\cos{\alpha_0},\sin{\alpha_0})\in S_{\alpha}^1$(or $S^1\backslash S_{\alpha}^1$), then $v(x)\geq v(l_{\alpha}(x))$ (or $v(x)\leq v(l_{\alpha}(x))$) for any $x\in S_{\alpha}^1$.
\end{cor}

\begin{cor}\label{cor2.2}
Let $v\in C(S^1,(0,\infty))$. Suppose that $v$ is separable and
 \begin{equation}\label{5}
v((\cos{\alpha},\sin{\alpha}))=v((\cos{(\alpha+\pi)},\sin{(\alpha+\pi)}))\ \mbox{for any } \alpha\in[0,2\pi).
\end{equation}
Then $v$ is a constant function on $S^1$.
\end{cor}
\begin{proof} By way of contradiction,
we assume that $v$ is not a constant function. In particular, $\max\limits_{S^1}v>\min\limits_{S^1}v.$ According to Lemma~\ref{lem2.1}, there exists $\alpha_0\in[0,2\pi)$ such that
$$\left\{\begin{array}{lll}
v((\cos{\alpha_0},\sin{\alpha_0}))=\max\limits_{S^1}v,\\
v((\cos{(\alpha_0+\pi)},\sin{(\alpha_0+\pi)}))=\min\limits_{S^1}v,
\end{array}\right.$$   a contradiction with \eqref{5}.  This completes the proof.
\end{proof}

\subsection{Separable functions in  spheres}
In this subsection, we study the axial symmetry and monotonicity of separable functions in high dimensional  spheres.

First we list the following  basic result,
which indicates that every element in $\mathbf{aff}(A)$ and $\mathbf{co}(A)$ are a combination of  at most $N+1$ points in $A$ if $A\subset \mathbb{R}^N,$ which is standard  and hence is omitted.

\begin{lemma}\label{lem2.2}
Let $A\subset \mathbb{R}^N$. Then we have the following results.
\begin {itemize}
\item [{\rm (i)}]
$\mathbf{aff}(A)=\mathbf{aff}_N(A):=\{\sum\limits^{N+1}_{i=1}\lambda_ix^i| x^i\in A, \lambda_i\in \mathbb{R}
\mbox{ and } \sum\limits^{N+1}_{i=1}\lambda_i=1\}$.
\item [{\rm (ii)}]
$\mathbf{co}(A)=\mathbf{co}_{N}(A):=\{\sum\limits^{N+1}_{i=1}\lambda_ix^i| x^i\in A, \lambda_i\in [0,1]
\mbox{ and } \sum\limits^{N+1}_{i=1}\lambda_i=1\}.$
\end {itemize}
Hence,  $\mathbf{co}(A)$ is bounded and closed if $A$ is bounded and closed.
\end{lemma}

Now we introduce the definition of separable functions in  spheres $S\subset \mathbb{R}^N$.
\begin{defn}\label{defn1}
Assume $2\leq k \leq N$ and  $S\subset \mathbb{R}^N$ is   a $k-1$ dimensional sphere. We say $u\in C(S,\mathbb{R})$ is
separable, if for any open half space $H\subset \mathbb{R}^N$ with $x^*\in \partial H$ and $\sigma_Hx\in S$ for all
$x\in S,$ there holds
$$\mbox{either $u(x)\geq u(\sigma_Hx)$ for all $x\in H\bigcap S$ or $u(x)\leq u(\sigma_Hx)$ for all $x\in H\bigcap S$.} $$
Here $x^*$ is the center of the ball $\mathbf{co } (S) $.
\end{defn}
We can also define separable functions in   spheres in another way.
\begin{defn}\label{defn2}
Assume $2\leq k \leq N$ and  $S\subset \mathbb{R}^N$ is   a $k-1$ dimensional sphere. We say $u\in C(S,\mathbb{R})$ is
separable, if there exist $r>0$, $b\in \mathbb{R}^N$, and $M\in O(N)$ such that
$$MS+b=rS^{k-1}\times \{0_{N-k}\}\subset \mathbb{R}^N, \ MV+b=\mathbb{R}^{k}\times \{0_{N-k}\}\subset \mathbb{R}^N$$
and  $\tilde{u}$ is separable in $S^{k-1}\subseteq \mathbb{R}^k$ in the sense of Definition~\ref{defn1}. Here
\begin{equation}\label{44}
V:=\mathbf{aff}(S)=\{\sum^{k+1}_{i=1}\lambda_ix^i| x^i\in S, \lambda_i\in \mathbb{R} \mbox{ and } \sum^{k+1}_{i=1}\lambda_i=1\},
\end{equation}
\begin{equation}\label{45}
\tilde{u}: S^{k-1}\ni x\mapsto u_M((rx,0_{N-k})-b)=u(M^{-1}((rx,0_{N-k})-b))\in (0,\infty).
\end{equation}
\end{defn}

It is obvious that Definition \ref{defn1} is equivalent to  Definition \ref{defn2}.  As a result, we say $u\in C(S^{N-1},(0,\infty))$ be  a  separable function, however, we don't have to emphasize the way we use the definition.

The next lemma is vital  to
investigate some basic properties of separable functions in spheres.

\begin{lemma}\label{lem2.3}
Let $u\in C(S^{N-1},(0,\infty))$ be  a  separable function in  $S^{N-1}$ and $V$ be a $k\in[1,N-1]$ dimensional hyperplane.  If $(V\bigcap S^{N-1})^{\#}>1$, then the following statements are true:
\begin {itemize}
\item[\rm{(i)}]
 $V\bigcap S^{N-1}$ is a $k-1$ dimensional sphere;
\item[\rm{(ii)}]  $u|_{S^{N-1}\bigcap V}$ is separable in $S^{N-1}\bigcap V$.
\end {itemize}
\noindent Here  $(V\bigcap S^{N-1})^{\#}$ represents the cardinality of elements contained in $V\bigcap S^{N-1}$.
\end{lemma}
\begin{proof}
(i) Take $b\in V$. Then $V-b$ is a $k$ dimensional linear subspace and thus there exists $M\in O(N)$ such that
$M(V-b)=\mathbb{R}^k\times \{0_{N-k}\}\subset\mathbb{R}^N$,
which implies that $M(V)=\mathbb{R}^k\times \{0_{N-k}\}+Mb.$
Let $Mb:=(a_1,a_2,\cdots,a_N).$ Then $MV=\mathbb{R}^k\times\{(a_{k+1},\cdots,a_N)\}.$ Note that  $MS^{N-1}=S^{N-1}$
and
 $$V\bigcap S^{N-1}=M^{-1}((\mathbb{R}^k\times\{(a_{k+1},\cdots,a_N)\})\bigcap S^{N-1}).$$
Hence, it suffices to prove that $(\mathbb{R}^k\times\{(a_{k+1},\cdots,a_N)\})\bigcap S^{N-1}$
is a $k-1$ dimensional sphere. Indeed, we may conclude that $a^2_{k+1}+\cdots+a^2_N<1.$
Then
\begin{displaymath}
\begin{array}{lll}
(\mathbb{R}^k\times\{(a_{k+1},\cdots,a_N)\})\bigcap S^{N-1}\\
=\{(x_1,\cdots,x_N)\in \mathbb{R}^N|x^2_1+\cdots+x^2_k+a^2_{k+1}+\cdots+a^2_N=1\}\\
=\{(x_1,\cdots,x_k)\in \mathbb{R}^k|x^2_1+\cdots+x^2_k=1-a^2_{k+1}-\cdots-a^2_N\}\times\{(a_{k+1},\cdots,a_N)\}
\end{array}
\end{displaymath} is a $k-1$ dimensional sphere.
So, the proof of (i) is complete.

(ii) By  (i), we see $V\bigcap S^{N-1}$ is a $k-1$ dimensional sphere whose center is denoted by $x^*_0.$  Then the vector $\overrightarrow{Ox_0^*}\bot {V}.$
Fix an open half space $H\subset \mathbb{R}^N$ with $x_0^*\in \partial H$ and $\sigma_H(V\bigcap S^{N-1})\subseteq  V\bigcap S^{N-1}$. It follows that the vector $\overrightarrow{x\sigma_H x}\bot \partial H$ and
$ \overrightarrow{x\sigma_H x}\bot \overrightarrow{Ox_0^*}$ for all $x\in V\bigcap S^{N-1}$. This, combined with $dim(\partial H)=N-1,$ implies that $\overrightarrow{Ox_0^*}// \partial H$. By $x_0^*\in \partial H$, we deduce that $O\in \partial H$. Applying the fact
that $u$ satisfies separability in $S^{N-1}$, we have
$$\mbox{either $u(x)\geq u(\sigma_Hx)$ for all $H\bigcap S^{N-1}$ or $u(x)\leq u(\sigma_Hx)$ for all $H\bigcap S^{N-1}$.}$$
In particular,
either $u(x)\geq u(\sigma_Hx)$ for all $H\bigcap (S^{N-1}\bigcap V)$ or $u(x)\leq u(\sigma_Hx)$ for all $H\bigcap (S^{N-1}\bigcap V)$, that is, the statement (ii) holds.
\end{proof}

The following is devoted to the proof of symmetry and monotonicity of separable functions in $S^{N-1}$.
\begin{lemma}\label{lem2.4}
Let $N\geq 2$ and $u\in C(S^{N-1},(0,\infty))$. Assume that $u$ is nonconstant and  separable in $S^{N-1}$. Then  $u$ is
 axially symmetric and monotone in $S^{N-1}$. To be precise, there exist $M\in O(N)$  and $h_1,h_2\in[-1,1]$ such that
 \begin {itemize}
\item[\rm{(i)}] $h_1>h_2;$
\item[\rm{(ii)}] $u_M^{-1}(\max\limits_{S^{N-1}}u_M)=\{x\in S^{N-1}|x_N\geq h_1\}$ and
$$u_M^{-1}(\min\limits_{S^{N-1}}u_M)=\{x\in S^{N-1}|x_N\leq h_2\}.$$
\item[\rm{(iii)}] For any fixed $h\in [-1,1]$, $u_M|_{\{x\in S^{N-1}|x_N=h\}}$ is constant.
\item[\rm{(iv)}] $u_M (0_{N-2},\cos\alpha,\sin\alpha)$ is decreasing with respect to $\alpha\in[\frac{\pi}{2},\frac{3\pi}{2}].$
    \end {itemize}
\end{lemma}
\begin{proof}
Since  $u$ is not constant, we have $A=u^{-1}(\max\limits_{S^{N-1}}u)\neq \emptyset$ and $B=u^{-1}(\min\limits_{S^{N-1}}u)\neq \emptyset.$

 We shall finish the proof by the following two steps.

\noindent \textbf{Step 1.} We  prove that $A$ is a single set or an  $N-1$ dimensional spherical cap as well as $B$.

We shall argue it  by inductive  method.

It follows from Lemma \ref{lem2.1} that the conclusion holds when $N=2$.

We assume that  the conclusion holds for $2\leq N\leq k$.

Now we prove that the conclusion is also valid for   $N=k+1.$
Without loss of generality, we assume that $A$ is not a single set. Note that $A\neq S^{N-1}$.
Then by Lemma \ref{lem2.2}, $\mathbf{co}(A)=\{\sum\limits^{N+1}\limits_{i=1}\lambda_ix^i| x^i\in A, \lambda_i\in [0,1] \mbox{ and } \sum\limits^{N+1}\limits_{i=1}\lambda_i=1\}$ and $\mathbf{co}(A)$ is
a closed convex set. Clearly, $\mathbf{co}(A)\subsetneqq B_1$ and $A\subsetneqq\partial (\mathbf{co}(A)),$ where $B_1$ is the unit closed ball with the center at the origin.  Take $x^*=(x^*_1,\cdots,x^*_N)\in \partial (\mathbf{co}(A))\backslash A$. By using the theorem of the separation of convex sets in \cite[Chapter 3]{rudin}, we can find an $N$ dimensional open half space $H$ such that $x^*\in \partial H$, $\mathbf{co}(A)
\bigcap H=\emptyset,$ and thus $A
\bigcap H=\emptyset.$ Then there exists
$\tilde{M}\in O(N)$ and $b\in \mathbb{R}^N$ such that
$\hat{H}:=\tilde{M}H+b=\mathbb{R}^{N-1}\times(-\infty,0)$. Let us define an affine transformation $T:\mathbb{R}^{N}\ni x\mapsto \tilde{M}x+b\in \mathbb{R}^{N}$.
So $\partial (T(H))=\mathbb{R}^{N-1}\times\{0\}$ and $\mathbf{co}(T(A))\subset\mathbb{R}^{N-1}\times[0,\infty)$.

Next we  show
$A\bigcap \partial H\neq \emptyset.$ Otherwise, $A\bigcap \partial H=\emptyset.$ Then $cl (H)\bigcap A=\emptyset$  and  hence
by the convexity of $\mathbb{R}^N\setminus cl (H)$, we have $cl (H)\bigcap \textbf{co}(A)=\emptyset$, a contradiction with the fact that  $x^*\in \textbf{co} (A)\bigcap \partial H.$

Let $y^*=(y^*_1,\cdots,y^*_N)\in A\bigcap \partial H.$
In view of $x^*\in \partial H \bigcap(\partial (\mathbf{co}(A))\backslash A)$ and $y^*\in\partial H\bigcap A,$ we see that $x^*\neq y^*.$

Now we claim that
$A\bigcap \partial H\neq \{y^*\}.$ Suppose on the contrary that $A\bigcap \partial H= \{y^*\}.$ It follows that $(T(x^*))_N=(T(y^*))_N=0$ and $x_N>0$ for any $x=(x_1,\cdots,x_N)\in T(A)\backslash\{T(y^*)\}.$ By $T(x^*)\in T(\mathbf{co}(A)),$  there  exists $x^i\in T(A)$ and $\lambda_i\in[0,1]$ with $\sum\limits^{N+1}\limits_{i=1}\lambda_i=1$
such that $T(x^*)=\sum\limits^{N+1}\limits_{i=1}\lambda_ix^i$ and so $0=(T(x^*))_N=\sum\limits^{N+1}\limits_{i=1}\lambda_ix_N^i\geq 0.$ Hence  $\lambda_ix_N^i=0$ for all $i=1,\cdots,N+1.$ Clearly, $\lambda_i=0$ or $x^i_N=0$
 for all $i=1,\cdots,N+1,$   which, together with  $T(A)\bigcap  (\mathbb{R}^{N-1}\times\{0\})=T(A)\bigcap \partial (T(H))= \{T(y^*)\}$, implies  $T(x^*)=T(y^*)$ and thus $x^*=y^*$, a contradiction to $x^*\neq y^*.$
 Therefore, the claim holds and thus there exists $y^{**}\in (A\bigcap \partial H)\backslash\{y^*\}.$

 Let $S=\partial H\bigcap S^{N-1}.$ Clearly, $y^*,y^{**}\in A\bigcap S$. By Lemma \ref{lem2.3}, $S$ is an $N-2$ dimensional sphere and $u$ is
 separable in $S.$  Let $\tilde{x},\tilde{r}$ be the center and radius of S, respectively, and let us define $\tilde{T}:S^{N-2}\ni z\mapsto T^{-1}(T(\tilde{x})+\tilde{r}(z,0))\in S$
and $\tilde{u}:S^{N-2}\ni z\mapsto  u(\tilde{T}(z))\in (0,\infty)$. Then we easily see that $\tilde{T}^{-1}(y^*),\tilde{T}^{-1}(y^{**})\in S^{N-2}\bigcap \tilde{u}^{-1}(\max\limits_{S^{N-2}}\tilde{u})$  and $\tilde{u}$ is
 separable in $S^{N-2}.$ By applying the inductive hypothesis to $\tilde{u}|_{S^{N-2}}$, we see that
 $\tilde{u}^{-1}(\max\limits_{S^{N-2}}\tilde{u})$ is an $N-2$ dimensional sphere cap and hence $u^{-1}(\max\limits_{S}u)$  is an $N-2$ dimensional sphere cap denoted by  $S^*.$

 Without loss of generality, we can assume, in the remaining proof, that there exist $h\in(-1,1)$ and $\delta\in[-\sqrt{1-h^2},\sqrt{1-h^2})$ such that  $ H=\mathbb{R}^{N-1}\times(-\infty,h),$ $A\subset \mathbb{R}^{N-1}\times[h,\infty), $
 $S=\{x\in S^{N-1}|x_N=h\}$ and
   $u^{-1}(\max\limits_{S}u)=\{x\in S|x_{N-1}\geq \delta\}.$

   Next we show $\delta=-\sqrt{1-h^2}$, that is, $S^*=S.$
  Otherwise $|\delta|<\sqrt{1-h^2}$. Let
  $$z^*_+=0_{N-3}\times (\sqrt{1-h^2-\delta^2},\delta,h),\ z^*_-=0_{N-3}\times(-\sqrt{1-h^2-\delta^2},\delta,h).$$
 In addition, for $\epsilon\geq h$, let $z^\epsilon_+=0_{N-2}\times (-\sqrt{1-\epsilon^2},\epsilon)$  and $z^\epsilon_-$ be the point at which  $S^{N-1}$ intersects the line containing $z^\epsilon_+$ and the point $0_{N-2}\times(\delta,h)$ .
Let us define $$f_{\pm}:[h,1]\ni \epsilon\mapsto u(z^\epsilon_{\pm})\in(0,\infty).$$ It is easy to check that $f_{+}(h)<\max\limits_{S}u, $ $(z^\epsilon_{-})_N<\delta$ for all $\epsilon\in (h,1)$, and $f_{+}$ is continuous and $f_{\pm}(h)<\max\limits_{S}u.$ So there exists $\epsilon^*>h$ such that
  \begin{equation}\label{ll}
   u(z^{\epsilon^*}_{\pm})=f_{\pm}(\epsilon^*)<\max\limits_{S}u=u(z^*_{\pm}).
   \end{equation}
  Notice that the line segments $\overline{z^*_{+}z^*_{-}}$ and $\overline{z^{\epsilon^*}_{+}z^{\epsilon^*}_{-}}$ are coplanar (see Figure \ref{figure4.jpg}).
   \begin{figure}[!ht]
  \center
  \includegraphics[height=5cm,width=6cm]{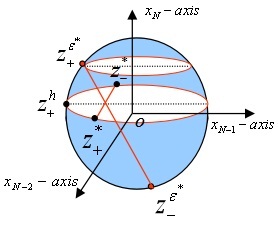}
  \caption{schematic diagram  for the partial  proof of \textbf{Step 1}}\label{figure4.jpg}
\end{figure}

\noindent Let $$V=\mathbf{aff} (\{z^*_{\pm},z^{\epsilon^*}_{\pm}\})
  =\{\lambda_1z^*_{+}+\lambda_2z^*_{-}+\lambda_3z^{\epsilon^*}_{+}
  +\lambda_4z^{\epsilon^*}_{-}
  |\lambda_i\in\mathbb{R},\sum_{i=1}^{4}\lambda_i=1\}.$$
  Then $\mathbf{dim} V=2$  and
   $\max\limits_{V\bigcap S^{N-1}}u=u(z^*_{\pm})>u(z^{\epsilon^*}_{\pm})$. By applying Lemma \ref{lem2.3} and Lemma \ref{lem2.1}, we can obtain that $V\bigcap S^{N-1}$ is a
  circle and $u^{-1}(\max\limits_{V\bigcap S^{N-1}}u)$ is an arc $\Lambda $ containing
   $z^*_{\pm}$.  This implies that $z^{\epsilon^*}_{+}\in \Lambda $ or $z^{\epsilon^*}_{-}\in \Lambda$, a contradiction  to \eqref{ll}.
  Hence $S^*=S$, that is,
  $$\{x\in S^{N-1}|x_N=h\}\subset A\subset\{x\in S^{N-1}|x_N\geq h\}:=A^*.$$

  Now we shall prove $A=A^*.$ We argue it by contradiction   as follows.
Let $x^*=(x^*_1,\cdots,x^*_N)\in A^*\backslash A,$ $w_{\pm}^*=(0_{N-1},{\pm}\sqrt{1-h^2},h)$, and
$\tilde{V}=\mathbf{aff}(\{w_+^*,w_-^*,x^*\})$.
Then
  $w_+^*,w_-^*\in A\cap \partial H$, $x^*\notin A\setminus cl(H)$, $\mathbf{dim} (\tilde{V})=2$,  and hence $\tilde{V}\cap S^{N-1} \cap H \neq \emptyset$ due to $\mathbf{dim} (\partial H)=N-1$ (see Figure \ref{figure5.jpg}).
  \begin{figure}[!ht]
  \center
  \includegraphics[height=5cm,width=6cm]{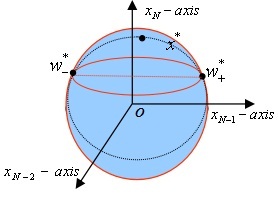}
  \caption{schematic diagram  for the partial  proof of \textbf{Step 1}}\label{figure5.jpg}
\end{figure}
By applying Lemma \ref{lem2.1}, we may obtain that $A\bigcap \tilde{V}\bigcap S^{N-1}$ is an arc $\Gamma$ containing
  $w_+^*,w_-^*.$ It follows from $\tilde{V}\cap S^{N-1} \cap H \neq \emptyset$ that $x^*\in \tilde{V}\cap S^{N-1} \setminus  H \subseteq\Gamma \subseteq A$, a contradiction to $x^*\notin A.$ As a result, we obtain that $A=A^*$ is  an $N-1$ dimensional sphere cap.

By  applying the above discussions to $1+
  \max\limits_{S^{N-1}}u-u$, we obtain that $B$ is a single point or an $N-1$ dimensional spherical cap. This completes the proof of Step 1.

  We denote the centers of two sphere caps  $A$ and  $B$ by $a^*$ and $b^*$, respectively.  We next verify that  $a^*$, $b^*$, and the origin $O$ are collinear. Otherwise, there exists an $N$ dimensional open half space $\tilde{H}$ such that  $a^*,b^*\in \tilde{H}$. Let $V^*=\mathbf{aff} (\{a^*, b^*, O\})$. Then $V^*$ is a two dimensional plane  and  by Lemma \ref{lem2.3},  $V^*\bigcap S^{N-1}$ is a  circle, and  $u|_{V^*\bigcap S^{N-1}}$ is nonconstant and  separable. Thus, it follows from the proof of Step 3 in Lemma \ref{lem2.1} that
  $a^*, b^*, O$ must be collinear, a contradiction. Since  $a^*$, $b^*$, and the origin $O$ must be collinear, we know that there exists $M\in O(N)$ such that $M(a^*)=(0_{N-1},1)$, $M(b^*)=(0_{N-1},-1)$, and hence $u_M$ satisfies (i) and (ii).

\noindent \textbf{Step 2.}  In this step, we shall prove (iii) and (iv).

We shall finish the proof by distinguishing two cases.

\item{\bf{ Case 1.}} $N=2$.\\
In this case, (iii)  and (iv) follow from Lemma \ref{lem2.1}.

\item{\bf{ Case 2.}} $N\geq3$.

\noindent (iii) Fix $h\in(-1,1).$  Then $\{x\in S^{N-1}|x_N=h\}$ is an $N-2$ dimensional sphere.

Letting $\bar{x},\bar{\bar{x}}\in\{x\in S^{N-1}|x_N=h\}$ be any  pair of symmetric points with respect to $(0_{N-1},h)$, we easily see that
  $W=\mathbf{aff}(\{\bar{x},\bar{\bar{x}},(0_{N-1},1)\})$ is a two dimensional plane and $u_M|_{W\bigcap S^{N-1}}$ is nonconstant.  Thus, by Lemma \ref{lem2.3},  $W\bigcap S^{N-1}$ is a  circle and $u_M|_{W\bigcap S^{N-1}}$ is separable. Note that by (ii),  $(0_{N-1},1)$ and $(0_{N-1},-1)$ are  centers of the arcs $W\bigcap S^{N-1}\bigcap u_M^{-1}(\max\limits_{S^{N-1}}u)$ and  $W\bigcap S^{N-1}\bigcap u_M^{-1}(\min\limits_{S^{N-1}}u)$, respectively.  By  applying Lemma \ref{lem2.1} to $u_M|_{W\bigcap S^{N-1}}$ under some affine transformation, we know that $u_M|_{W\bigcap S^{N-1}}$ is axial symmetric with respect to $x_N-$ axis. In particular, we have
$u_M(\bar{x})=u_M(\bar{\bar{x}}).$

Let  $\bar{x}^*,\bar{\bar{x}}^*\in\{x\in S^{N-1}|x_N=h\}$. Then $W^*=\mathbf{aff}(\{\bar{x}^{*},\bar{\bar{x}}^{*},
   (0_{N-1},h)\})$  is a two dimensional plane and hence by Lemma \ref{lem2.3},  $W^*\bigcap S^{N-1}$ is a  circle and  $u_M|_{W^*\bigcap S^{N-1}}$ is separable.
  These, together with Corollary~\ref{cor2.2} and the fact that $  u_M(x)=u_M(y)$  whence $x,y\in W^*\bigcap S^{N-1}$ are given symmetric pairs with respect to $(0_{N-1},h)$, implies that $u_M|_{W^*\bigcap S^{N-1}}$ is constant. In particular, $u_M(\bar{x}^*)=u_M(\bar{\bar{x}}^*)$. So by the  arbitrariness of $\bar{x}^*,\bar{\bar{x}}^*\in\{x\in S^{N-1}|x_N=h\}$,  we get (iii).

 \noindent (iv) By (ii), (iii), and by applying  Lemma \ref{lem2.1} (iv) to $u_M(0_{N-2},\cdot)|_{S^1}$, we easily see that $u_M (0_{N-2},\cos\alpha,\sin\alpha)$ is decreasing with respect to $\alpha\in[\frac{\pi}{2},\frac{3\pi}{2}].$

The proof is completed.
\end{proof}

\begin{cor}\label{cor2.3}
 Let $H$  be an open half space in $\mathbb{R}^N$ with the origin $O\in \partial H$. Under the assumptions of  Lemma~\ref{lem2.4}, we have the following statements:
 \begin {itemize}
\item[\rm{(i)}] If  $M^{-1}(0,0,\cdots,0,1)\in H,$ then $u(x)\geq u(\sigma_Hx)$ for any $x\in H\bigcap S^{N-1}$;
\item[\rm{(ii)}] If $M^{-1}(0,0,\cdots,0,1)\in\mathbb{R}^N\backslash cl(H) ,$ then $u(x)\leq u(\sigma_Hx)$ for any $x\in H\bigcap S^{N-1}$.
\end{itemize}
\end{cor}

\subsection{Separable functions in balls}
In this subsection, we consider the axial symmetry and monotonicity of separable functions in high dimensional balls.

We first introduce the definition of separable functions in $B_R.$
\begin{defn}
A function $u:B_R\to\mathbb{R}$ is said to be separable if
for any  open half-space  $H\subset \mathbb{R}^N$ with $O\in \partial H$,
\begin{equation}\label{11}
\mbox{ either $u(x)\geq u(\sigma_Hx)$ for all $x\in H\bigcap B_R $ or
$u(x)\leq u(\sigma_Hx)$ for all $x\in H\bigcap B_R$. }
\end{equation}
\end{defn}
\begin{thm}\label{thm2.1}
 Let  $u\in C(B_R,(0,\infty))$ be a separable function. If $u$ is not radially symmetric with respect to the origin $O$, then there exists  $M\in O(N)$ such that the following statements are true:
 \begin {itemize}
\item [{\rm (i)}]
(Axial symmetry). For any  $\alpha\in(0,R]$  and $h\in [-\alpha,\alpha],$   ${u_{M}}|_{\{x=(x_1,x_2,\cdots,x_N)\in S_\alpha^{N-1}| x_N=h\}}$ is  constant, that is,  $u_M$  is axially symmetric with respect to  $x_N$-axis;
\item [{\rm (ii)}](Monotonicity). For any given  $\alpha\in(0,R],$
$u_{M}(0_{N-2},\alpha\cos\theta,\alpha\sin\theta)$ is decreasing with respect to
$\theta\in[\frac{\pi}{2},\frac{3\pi}{2}]$.
 \end {itemize}
\end{thm}
\begin{proof}
Fix $\alpha\in (0,R]$. By applying Lemma~\ref{lem2.4}
to $u(\alpha \cdot)|_{S^{N-1}}$, we know that  $u^{-1}(\max\limits_{S_\alpha ^{N-1}} u)\bigcap S_\alpha^{N-1}$  are a single set or a spherical cap of  $S_\alpha^{N-1}$. Let us  write $x_*^{\alpha}$ for the centers of $u^{-1}(\max\limits_{S_\alpha^{N-1}} u)\bigcap S_\alpha^{N-1}$.

We claim  that there exists a radial with the peak at $O$ passing through $x_*^{\alpha}$ and $x_*^{\beta}$ for any $\alpha,\beta\in(0,R]$. Indeed, if either $u|_{S_\alpha^{N-1}}$ or  $u|_{S_\beta^{N-1}}$ is  a constant function, then
$$u^{-1}(\max_{S_\alpha^{N-1}} u)\bigcap S_\alpha^{N-1}=S_\alpha^{N-1}
\mbox{ or }  u^{-1}(\max_{S_\beta^{N-1}} u)\bigcap S_\beta^{N-1}=S_\beta^{N-1}.$$
So we can re-select $x_*^\alpha$ ( or $x_*^\beta$) belonging to $Ox_*^\beta \cap S_\alpha^{N-1} $ (or $Ox_*^\alpha \cap S_\beta^{N-1}$).

If neither $u|_{S_\alpha^{N-1}}$ nor $u|_{S_\beta^{N-1}}$ is a  constant  function, then  $x_*^{\alpha}$, $ x_*^{\beta}$
 are unique centers of $u^{-1}(\max\limits_{ S_\alpha^{N-1}} u)\bigcap S_\alpha^{N-1}$, $u^{-1}(\max\limits_{S_\beta^{N-1}} u)\bigcap S_\beta^{N-1}$, respectively.
Suppose that  $x_*^{\alpha}$, $ x_*^{\beta}$ are not
in same radial with the peak at $O$.  Then there exists  an open half space $H$ in $\mathbb{R}^N$ such that the origin $O\in \partial H$ , $x_*^{\alpha}\in H$, and  $ x_*^{\beta}\notin cl(H)$. It follows from Corollary  \ref{cor2.3} that $u(x)\geq u(\sigma_Hx)$ for any $x\in H\bigcap S_\alpha^{N-1}$ and $u(x)\leq u(\sigma_Hx)$ for any $x\in H\bigcap S_\beta^{N-1}$. These, together with  the separability of $u$, implies  either $u(x)=u(\sigma_Hx)$ for any $x\in H\bigcap S_\alpha^{N-1}$ or $u(x)= u(\sigma_Hx)$ for any $x\in H\bigcap S_\beta^{N-1}$. Thus,  either $u|_{ S_\alpha^{N-1}}$ or  $u|_{ S_\beta^{N-1}}$ is  a constant function, a contradiction. So the claim is true.

By applying Lemma~\ref{lem2.4} to $u(\alpha \cdot)|_{ S^{N-1}},$ there exists  $M_\alpha\in O(N)$ such that ${u_{M_\alpha}}|_{\{x=(x_1,x_2,\cdots,x_N)\in S_\alpha^{N-1}| x_N=h\}}$ is a constant and $u_{M_\alpha}(0_{N-2},\alpha\cos\theta,\alpha\sin\theta)$ is decreasing with respect to
$\theta\in[\frac{\pi}{2},\frac{3\pi}{2}]$.

 We may assume $u|_{ S_{\alpha_0}^{N-1}}$ is not constant for some $\alpha_0\in(0,R)$ since $u$ is not radially symmetric with respect to the origin $O$.

 In the following,  we shall prove $M_\beta M_{\alpha_0}^{-1} (\mathbb{R}^{N-1}\times \{\eta\})\subseteq \mathbb{R}^{N-1}\times \{\eta\}$ for all $\beta\in (0,R]$ and $\eta \in \mathbb{R}$.

 If $u|_{ S_\beta^{N-1}}$ is constant, then we may re-select $M_\beta={M_{\alpha_0}}$ such that ${u_{{M_{\alpha_0}}}}|_{\{x=(x_1,x_2,\cdots,x_N)\in S_\beta ^{N-1}| x_N=h\}}$ is a constant and $u_{{M_{\alpha_0}}}(0_{N-2},\beta\cos\theta,\beta\sin\theta)$ is decreasing with respect to $\theta\in[\frac{\pi}{2},\frac{3\pi}{2}]$.  So, $M_\beta M_{\alpha_0}^{-1} (\mathbb{R}^{N-1}\times \{\eta\})= \mathbb{R}^{N-1}\times \{\eta\}$ for all $\eta \in \mathbb{R}$.

 Now suppose that $u|_{ S_\beta^{N-1}}$ is not constant. Then $x_*^{\alpha_0}=\alpha_0 M_{\alpha_0}^{-1}(0_{N-1},1)$ and  $x_*^{\beta}=\beta M_{\beta}^{-1}(0_{N-1},1)$. This, combined with the fact $x_*^{\alpha}$, $ x_*^{\beta}$ are
in same radial with the peak at $O$, gives  $M_{\alpha_0}^{-1}(0_{N-1},1)=M_{\beta}^{-1}(0_{N-1},1)$, that is, $M_\beta M_{\alpha_0}^{-1}  (0_{N-1},1)=(0_{N-1},1)$. As a result, $M_\beta M_{\alpha_0} ^{-1}(\mathbb{R}^{N-1}\times \{\eta\})\subseteq \mathbb{R}^{N-1}\times \{\eta\}$ for all $\eta \in \mathbb{R}$.

In view of the choices of $M_{\alpha_0}, M_\beta$ and the fact that $M_\beta M_{\alpha_0}  ^{-1} (\mathbb{R}^{N-1}\times \{\eta\})\subseteq \mathbb{R}^{N-1}\times \{\eta\}$ for all $\eta \in \mathbb{R}$, we easily see $u_{M_{\beta}}|_{ S_\beta^{N-1}}=u_{M_{\alpha_0}}|_{S_\beta^{N-1}}$. To sum up, we may re-select $M:=M_\beta={M_{\alpha_0}}$ for all $\beta\in (0,R]$ with statements (i) and (ii).

The proof is completed.
\end{proof}

It is easy to obtain the following results that the axially symmetric and monotone functions in balls are   separable.
\begin{thm}\label{thm2.2}
Let  $u\in C(B_R,(0,\infty))$ be an axially symmetric function. Suppose that there exists $M\in O(N)$ such that for any given  $\alpha\in(0,R],$
$u_{M}(0_{N-2},\alpha\cos\theta,\alpha\sin\theta)$ is decreasing with respect to
$\theta\in[\frac{\pi}{2},\frac{3\pi}{2}]$. Then $u$ is separable in $B_R.$
\end{thm}

By Theorem~\ref{thm2.2}, the separability is  equivalent to  the axial symmetry and monotonicity for a given function in balls. This observation  enables us  to give following examples that   separable functions  in balls  may be only axially symmetric  but not radially symmetric.

\begin{example}\label{rem2.1}

Let $R>0$, let $g\in C([-R,R],(0,\infty))$ be  a nonconstant and nonincreasing function, and let $h\in C([0,R],[0,\infty))$ with
$h(R)=0$ and $h([0,R))\subset (0,\infty).$  Define $u:B_R\to \mathbb{R}$ by
$$u(x_1,x_2,\cdots,x_N)=g(x_N)h(\sqrt{{x_1}^2+{x_2}^2+\cdots+{x_N}^2}).$$
It is easy to check that $u$ satisfies the separable property. However,
 $u$ is only axially symmetric
 with respect to the $x_N$-axis and  is not radially symmetric with respect to the origin $O$.
\end{example}

\section{Separable functions in whole space}

\par

In this section, based on the results obtained in Section 2, we shall show that a separable function in  $\mathbb{R}^N$ can imply its radial symmetry and monotonicity.

First we give the definition of separable functions in  $\mathbb{R}^N$.
\begin{defn}\label{defn3.1}
 A function $u\in C(\mathbb{R}^N,\mathbb{R})$ is called
separable if for any open half space $H\subset \mathbb{R}^N$, there holds
\begin{equation}\label{700}
\mbox{ either $u(x)\geq u(\sigma_Hx)$ for all $x\in H$ or $u(x)\leq u(\sigma_Hx)$ for all $x\in H$}.
\end{equation}
\end{defn}

Let $u\in C(\mathbb{R}^N,\mathbb{R})$. A line $L\subset \mathbb{R}^N$ is  a symmetry axis of $u$ if and only if for any given $\alpha>0$, $z\in L\cap V$,  and $N-1$ dimensional hyperplane $V$ with $L\bot V$,   $u|_{S_\alpha^{N-1}(z)\cap V}$ is constant.

In the following lemma, we give some properties of separable functions in  $\mathbb{R}^N$.

\begin{lemma}\label{lem3.1}
Let $u\in C(\mathbb{R}^N,(0,\infty))$ be a separable function and let $\mathcal{L}$ be the set of all  the symmetry axes of $u$. Assume that  $u$ is not radially symmetric in $\mathbb{R}^N$. Then  the following statements are true.
\begin {itemize}
\item[\rm{(i)}] For any $x\in\mathbb{R}^N$, there exists an unique  $L_x:=L(x)\in \mathcal{L}$ such that  $x\in L_x$, and hence $L_x=L_y$ whence $y\in L_x$.
\item[\rm{(ii)}] For any $x,y\in \mathbb{R}^N$, there holds
either $L_x=L_y$ or $L_x//L_{y}$ (that is,
$L_x$ is parallel to $L_{y}$).
\end{itemize}
\end{lemma}
\begin{proof}
(i). Fix $x\in\mathbb{R}^2$. Since $u$ is not a radially symmetric function, there exists $\alpha_0>0$ such that $\max\limits_{S_{\alpha_0}^{N-1}(x)}u>\min\limits_{S_{\alpha_0}^{N-1}(x)}u$. By applying Lemma~\ref{lem2.4} to $u(x+\alpha_0 \cdot)|_{S^{N-1}}$, we know  $u|_{S_{\alpha_0}^{N-1}(x)}$  is only an axially symmetric function, where $u(x+\alpha_0 \cdot)|_{S^{N-1}}:S^{N-1}\ni z\mapsto u(x+\alpha_0 z)\in (0,\infty)$.  Let   $L_x:=L(x)$ be the line containing the symmetry axis  of $u|_{S_{\alpha_0}^{N-1}(x)}$.

Now we prove $L_x\in \mathcal{L}$.  Fix $\alpha>0$, $z\in L_x\cap V$,  and a $N-1$ dimensional hyperplane $V$ with $L_x\bot V$.
By applying Theorem~\ref{thm2.1} to $u|_{B_{\max \{ \alpha_0,\alpha+\|x-z\|\}}(x)}$, we obtain that  $L_x$ is a unique symmetry axis of $u|_{B_{\max \{ \alpha_0,\alpha+\|x-z\|\}}(x)}$ and $u|_{ S_{\alpha}^{N-1}(z)\cap V}$ is constant, which implies
$L_x\in \mathcal{L}$.

By the uniqueness of symmetry axis through one point, we easily see  $L_x=L_y$  for any  $y\in L_x$. This completes of the proof of (i).

(ii)  Fix $x,y\in \mathbb{R}^N$. By (i), we only consider the case of  $y\notin L_x$.
We  prove it  by contradiction. Suppose on the contrary that
 $L_{x}$ is not parallel to $L_y$.

 We shall finish the proof by distinguishing two cases.	\\
\noindent\emph{Case 1}. $L_x\cap L_y\neq \emptyset $, that is,  $L_x $ and $ L_y$ are coplanar.

Take $x^*\in L_x\cap L_y$. Then $L_x, L_y$ are two different symmetry axes of $u$ through $x^*$, a contradiction with (i).

\noindent\emph{Case 2}. $L_x\cap L_y= \emptyset $, that is $L_x$ and $L_y$ are not coplanar.

Since $u$ is not radially symmetric, it follows from Lemma~\ref{lem2.4} that there exist
  positive constants $R_1$ and $R_2$ such that  both  $u|_{ S_{R_1}^{N-1} (x)}$ and $u|_{S_{R_2}^{N-1}(y)}$ are nonconstant and axially symmetric function with respect to the   $L_x$ and  $L_{y}$, respectively.

  Take   $x^*\in L_x\cap S_{R_1}^{N-1} (x),y^*\in L_y\cap S_{R_2 }^{N-1} (y)$ with
  $u(x^*)= \max\limits_{S_{R_1}^{N-1}(x)} u$ and
  $u(y^*)=\max\limits_{S_{R_2}^{N-1} (y)} u$.
  Then $\mathbf {dim}  (
 \mathbf{aff}(\{x,y,\frac{x^*+y^*}{2}\})) =2$, $\mathbf {dim}  (
 \mathbf{aff}(\{x,y,x^*,y^*\})) =3$, and thus there exists a hyperplane $\hat{H}\subseteq \mathbb{R}^N$ such that $\mathbf{aff}(\{x,y,x^*,y^*\})\times \hat{H}\subset \mathbb{R}^N$ and $\mathbf {dim} (\hat{H})=N-3$.

 Let $H$ be open half space $H\subset \mathbb{R}^N$ with $\partial H= \mathbf{aff}(\{x,y,\frac{x^*+y^*}{2}\})\times \hat{H}$.  Then $x,y\in \partial H$,  $x^*,y^*\notin \partial H$ and $\{x^*,y^*\}\setminus H \neq \emptyset$. Without loss of generality, we may assume that  $x^*\in H$
 and $y^*\notin cl(H)$.
  By the choices of $x^*$ and $y^*$,
 there exist $\tilde{x}^*\in H\bigcap S_{R_1}^{N-1} (x)$ and $\tilde{y}^*\in S_{R_2}^{N-1}(y)\setminus  cl(H)$ such that $u(\tilde{x}^*)>u(\sigma_H\tilde{x}^*)$ and $u(\tilde{y}^*)>u(\sigma_H\tilde{y}^*)$. Note that $\tilde{x}^*, \tilde{y}^*\in B_{\|x-y\|+2(R_1+R_2)}(y)$ and  $u|_{B_{\|x-y\|+2(R_1+R_2)}(y)}$ is nonconstant.
Hence by the   separability of $u$, we deduce a
 a contradiction.

To sum up, the proof is completed.
\end{proof}

In what follows, we describe monotonicity of even separable functions in $\mathbb{R}$, which is important to obtain the monotonicity of radial separable functions in $\mathbb{R}^N.$
\begin{lemma}\label{lem3.2}
Let $u\in C(\mathbb{R},(0,\infty))$, $u(x)=u(-x)$ and $\liminf\limits_{|x|\to\infty}u(x)=0.$
Suppose that for any  $x\in \mathbb{R},$
\begin{equation}\label{8}
\mbox{either $u(y)\geq u(2x-y)$ for all $y\geq x$ or
$u(y)\leq u(2x-y)$ for all $y\geq x.$ }
\end{equation}
Then $u$ is nonincreasing  on $[0,\infty)$.
\end{lemma}
\begin{proof}
Let
\begin{displaymath}
\begin{array}{lll}
I=\{\alpha\geq 0| \mbox{ there exists } x_{\alpha}>\alpha\mbox{ such that }u(x_\alpha)>u(2\alpha-x_\alpha)\},\\
J=\{\alpha\geq 0| \mbox{ there exists }x_{\alpha}>\alpha\mbox{ such that }u(x_\alpha)<u(2\alpha-x_\alpha)\},\\
K=\{\alpha\geq 0| u(x)\equiv u(2\alpha-x)\mbox{ for any } x\in\mathbb{R}\}.
\end{array}
\end{displaymath}
Obviously, $0\in K$ and $I\bigcup J\bigcup K=[0,\infty).$ By the continuity of  $u$,  for any $\alpha\in I$,  there exists $\delta_\alpha\in (0,\alpha)$  such that $x_{\alpha}>\beta$ and $u(x_\alpha)>u(2\beta-x_\alpha)$ for all
$\beta\in (\alpha-\delta_\alpha,\alpha+\delta_\alpha)$, that is, $(\alpha-\delta_\alpha,\alpha+\delta_\alpha)\subset I$. Hence, $I$ is an open set. Similarly,
$J$ is also an open set.  It suffices to prove  $J\bigcup K=[0,\infty)$, since  \eqref{8} and $J\bigcup K=[0,\infty)$  imply that $u$ is nonincreasing on $[0,\infty)$. If not, suppose $J\bigcup K \neq [0,\infty)$. Then $I\neq \emptyset$.
Note that $I\bigcup K\neq [0,\infty)$ since $I\cup K=[0,\infty)$ will yield a contradiction to $\liminf\limits_{|x|\to\infty}u(x)=0$.  Then
$I\neq \emptyset$ and $J\neq \emptyset.$ Since $I\bigcap J=\emptyset$ and $I,J$ are open sets, we have
$I\bigcup J\neq(0,\infty)$, and thus $K\backslash\{0\}\neq \emptyset.$ To be precise, there exists $\alpha^*\in (0,\infty)$ such that
 $u(x)=u(2\alpha^*-x)=u(x-2\alpha^*)$ for all $x\in\mathbb{R}.$ Hence $u$ is a periodic function in $\mathbb{R}$, which contradicts with  $\liminf\limits_{|x|\to\infty}u(x)=0.$
This proves the claim and hence the proof is completed.
\end{proof}

Now we are ready to prove radial symmetry and monotonicity of separable functions in  $\mathbb{R}^N$.

\begin{thm}\label{thm3.1}
Let $u\in C(\mathbb{R}^N,(0,\infty))$ be separable and
$\liminf\limits_{|x|\to\infty}u(x)=0$.
Then $u$ is  radially symmetric decreasing  with respect to some point, that is,
there exist $x^*\in \mathbb{R}^N$ and  a decreasing function $v:[0,\infty)\to(0,\infty)$ such that
$$u(x)=v(|x-x^*|)$$
and $\lim\limits_{r\to\infty}v(r)=0.$
\end{thm}
\begin{proof}
We shall argue it by contradiction. Suppose on the contrary that $u$ is not radially symmetric. Then by Lemma~\ref{lem2.4} and  Lemma \ref{lem3.1}-(i), there exists $M\in O(N)$ such that $u_M$ has  a unique  symmetry axis $x_N$-axis  through $O$ and $(0,0,\cdots,0,s)$ is the center of the spherical cap consisting of the maximum points of $u_M|_{S_s^{N-1}}$.

Now we claim that  for any $(x_1,x_2,\cdots,x_{N-1},x_N)\in\mathbb{R}^N,$ we have
\begin{equation}\label{46}
u_M(x_1,x_2,\cdots,x_{N-1},x_N)= u_M(0,0,\cdots,0,x_N).
\end{equation}
In fact, by taking $y^*=(\frac{x_1}{2},\frac{x_2}{2},\cdots,\frac{x_{N-1}}{2},0),$ and by applying  Lemma \ref{lem3.1}-(ii) to $u_M$, we conclude that $u_M$ is axially symmetric with respect to $L_{y^*}$ and
$L_{y^*}// x_N$-axis.
Hence the claim follows from the symmetry  pair $(x_1,x_2,\cdots,x_{N-1},x_N)$ and  $(0,0,\cdots,0,x_N)$ with respect to $L_{y^*}$.

Now, fix $s^*>0$.   In view of $u_M(0,0,\cdots,0,s^*)=\max u_M(S_{s^*}^{N-1})$, we have
$$u_M(0,0,\cdots,0,s^*)\geq u_M(\sqrt{{s^*}^2-\tilde{s}^2,}0,\cdots,0,\tilde{s})$$
for any $|\tilde{s}|\leq s^*$. This, combined with \eqref{46}, implies that
$$u_M(0,0,\cdots,0,s^*)\geq u_M(0,0,\cdots,0,\tilde{s}).$$
Hence $u_M(x_1,x_2,\cdots,x_{N-1},x_N)$ is a nondecreasing function with respect to $x_N\in(0,\infty)$, which contradicts with the assumption that $\liminf\limits_{|x|\to\infty}u_M(x)=0$.
So $u$ is radially symmetric with respect to some point in $\mathbb{R}^N$.

Finally, by Lemma ~\ref{lem3.2}, $u$ is a radially symmetric decreasing function.
This completes the proof.
\end{proof}

\section{Applications}
In this section,  we illustrate
our main results with the following  nonlocal Choquard equation,
\begin{equation}\label{30}
-\Delta u+u=\left(\int_{\mathbb{R}^N}\frac{|u(y)|^p}{|x-y|^{N-\alpha}}dy\right)|u|^{p-2}u,\ \ x\in \mathbb{R}^N
\end{equation}
with  $N\geq3, \alpha\in (0,N), \frac{N+\alpha}{N}< p<\frac{N+\alpha}{N-2}$.

For the generalized Choquard equation \eqref{30}, the existence and properties of solutions have been  widely considered. See \cite{LEH,LPL,tm,cp,wtyt,gmvj,gmv} and references therein. In particular, Moroz and Van
Schaftingen \cite{vv} obtained the separability, radial symmetry and monotonicity of all the positive ground
states  of \eqref{30}; Ma and Zhao \cite{MZ} proved
that positive solutions for \eqref{30} must be radially symmetric and
monotonically decreasing about some point under appropriate assumptions on $N,\alpha, p$  by using the
method of moving planes in integral form introduced by Chen  et al. \cite{clo}.

Let $H^1_0(B_R)$ be the usual Sobolev space with the standard norm
$\|u\|:=\left(\int_{B_R}(|\nabla u|^2+|u|^2)dx\right)^{\frac{1}{2}}.$
Let $\Omega \subset  \mathbb{R}^N$. For any $1\leq s<\infty,$ the norm on $L^s(\Omega)$
is denoted  by
$|u|_{L^s(\Omega)}:=\left(\int_{\Omega}|u|^sdx\right)^{\frac{1}{s}}.$

\subsection{ Choquard type equations in balls}
It is well known that
 when $N\geq3$ and $\alpha=2$, by  rescaling,
\eqref{30}  is equivalent to
\begin{equation}\label{15}
\left\{
\begin{array}{lll}
-\Delta u+u=w|u|^{p-2}u\ \mbox{in }\mathbb{R}^N,\\
-\Delta w=|u|^p\ \mbox{in }\mathbb{R}^N.
\end{array}\right.
\end{equation}
So the Dirichlet problem  in a ball $B_R$ is
\begin{equation}\label{16}
\left\{
\begin{array}{lll}
-\Delta u+u=w|u|^{p-2}u\ \mbox{in } B_R,\\
-\Delta w=|u|^p\ \mbox{in } B_R,\\
w=u=0\  \mbox{in }  \partial B_R.
\end{array}\right.
\end{equation}
It is clear  that by using Green's function (see \cite{els}), \eqref{16}
can  be rewritten as
\begin{equation}\label{17}
-\Delta u+u=\left(\int_{B_R}G(x,y)|u^p(y)|dy\right)|u|^{p-2}u,\ \ x\in B_R.
\end{equation}
Here
\begin{displaymath}
G(x,y)=\frac{1}{|y-x|^{N-2}}-\frac{1}{(\frac{|x|}{R}|y-\tilde{x}|)^{N-2}},\ (x,y\in B_R \mbox{ with } x\neq y),
\end{displaymath}
where $\tilde{x}$ is  the dual point of  $x$ with respect to $\partial B_R$ and can be
defined  by $\tilde{x}=\frac{R^2x}{|x|^2}$.

The existence of  positive ground states  for \eqref{17} can be obtained by using  variational methods. But the symmetry and monotonicity of the positive ground states  for \eqref{17} are very difficult to deal with. Now  the method of moving planes in integral form used in \cite{MZ} is not applicable to \eqref{17} and the main obstacle is  to establish the equivalence between the differential equation and the integral equation.
Moreover, the arguments  in \cite[Proposition 5.2]{vv} is also not valid because the origin $O$ is required to  belong to $\partial H$ in \eqref{17} for any half-space $H\subset \mathbb{R}^N$. But applying  our main results can lead to
axial  symmetry and monotonicity of all the  positive ground states  to \eqref{17}.

As usual, for $N\geq3$ and $p\in(\frac{N+2}{N},\frac{N+2}{N-2}),$ the corresponding energy functional  $I:H^1_0(B_R)\to \mathbb{R}$ associated to \eqref{17} is
\begin{equation}\label{18}
I(u)=\frac{1}{2}\int_{B_R}(|\nabla u|^2+|u|^2)dx-\frac{1}{2p}\int_{B_R}\int_{B_R}G(x,y)|u(y)|^p|u(x)|^p dxdy,
\end{equation}
due to the symmetry and positivity of $G(x,y)$ for $x,y\in B_R$ and $x\neq y$.
By Hardy-Littlewood-Sobolev inequality and Sobolev inequality, we have
\begin{displaymath}
\begin{array}{lll}
\int_{B_R}\int_{B_R}G(x,y)|u(y)|^p|u(x)|^p dxdy
&\leq&\int_{B_R}\int_{B_R}\frac{|u(y)|^p|u(x)|^p}{|x-y|^{N-2}}dxdy\\
&=&\int_{\mathbb{R}^N}\int_{\mathbb{R}^N}\frac{\chi_{B_R}(y)|u(y)|^p\chi_{B_R}(x)|u(x)|^p}{|x-y|^{N-2}}dxdy\\
&\leq&C|u|^{2p}_{L^{\frac{2Np}{N+2}}(B_R)}\leq C\|u\|^{2p},
\end{array}
\end{displaymath}
 where  $\chi_{B_R}$ denotes the characteristic  function on $\mathbb{R}^N$.
 It is easy to check that  $I\in C^1(H_0^1(B_R),\mathbb{R})$
and its Gateaux  derivative is given by
$$I'(u)v=\int_{B_R}(\nabla u\nabla v+uv)dx-\int_{B_R}\int_{B_R}G(x,y)|u(y)|^p|u(x)|^{p-2}u(x)v(x)dxdy$$
for any $v\in H^1_0(B_R).$
Recall that the critical points of $I$ are  solutions of \eqref{17} in the weak sense.
Let $c:=\inf\limits_{u\in\mathcal{N}}I(u)$, where
$\mathcal{N}=\{u\in H_0^1(B_R)\backslash{\{0\}}:I'(u)u=0\}.$
For simplicity of notations, we denote
$$\mathbb{D}(u)=\int_{B_R}\int_{B_R}G(x,y)|u(y)|^p|u(x)|^pdxdy.$$

The proof of the following properties of the Nehari manifold $\mathcal{N}$  is standard and hence is omitted here.
\begin{lemma}\label{lem4.1}
The following statements are true:
\begin {itemize}
\item[\rm(i)] $0\notin \partial\mathcal{N}$ and $c>0$;
\item[\rm(ii)] For any \,$u\in H^1_0(B_R)\backslash\{0\}$\,,
there exists a unique  $t_u\in(0,\infty)$ such that $t_uu\in \mathcal{N}$ and
 $t_u=\left(\frac{\|u\|^2}{\mathbb{D}(u)}\right)^\frac{1}{2p-2}.$
Furthermore,
\begin{equation}\label{19}
I(t_u u)=\sup\limits_{t>0}I(tu)=(\frac{1}{2}-\frac{1}{2p})\left(\frac{\|u\|^2}{\mathbb{D}^{\frac{1}{p}}(u)}\right)^{\frac{p}{p-1}};
\end{equation}
\item[\rm(iii)] $c=\inf\limits_{u\in\mathcal{N}} I(u)=\inf\limits_{u\in H^1_0(B_R)\backslash\{0\}}\sup\limits_{t>0}I(tu)$.
\end{itemize}
\end{lemma}

By using Nehari  maifold methods,  we can obtain the existence of ground states of \eqref{17} in $H^1_0(B_R)$.
Recall that $u\in H^1_0(B_R)$ is said to be a ground state of \eqref{17}, if $u$ solves \eqref{17} and
minimizes the energy functional associated with \eqref{17} among all possible nontrivial solutions.
Furthermore, by standard elliptic regularity estimate and strong maximum principle, we conclude that any ground state of \eqref{17} belongs to $C^{2}(\bar{B}_R)$, and $u>0$ or  $u<0$ in $B_R$.
Since the nonlocal term of \eqref{17} has some strong  symmetrizing effect, by using the minimality property of the ground states, we shall deduce the separability property of the positive ground states.
We start with the following lemma.

\begin{lemma}\label{lem4.2}
Let $H$ be an open half space in $\mathbb{R}^N$ with $0\in\partial H$. Then the following statements are true:
\begin{itemize}
\item[\rm{(i)}] $G(x,y)=G({\sigma_H}x,{\sigma_H}y)$ for any $x,y\in B_R$ and $x\neq y$;
\item[\rm{(ii)}] $G(x,{\sigma_H}y)=G({\sigma_H}x,y)$ for any $x,y\in B_R$ and $x\neq y$;
\item[\rm{(iii)}] $G(x,y)\geq G({\sigma_H}x,y)$ for any $x,y \in H\bigcap B_R$ and  $x\neq y$.
\end{itemize}
\end{lemma}
\begin{proof}
It is easy to check (i) and (ii). We  shall prove (iii).
Set $$a=|x-y|, \ \widetilde{a}=|y-\sigma_Hx|,\ b=\frac{|x|}{R}|y-\widetilde{x}|,
\ \widetilde{b}=\frac{|\sigma_Hx|}{R}|y-\widetilde{\sigma_Hx|},$$
 where $\widetilde{x}$ and  $\widetilde{\sigma_Hx}$ represent the dual points of  $x$ and  $\sigma_Hx$ with respect to  $\partial B_R$. We will split  the proof into two steps.

\noindent\textbf{Step 1.} We claim  that $\frac{\tilde{a}}{a}\ge \frac{\tilde{b}}{b}\ge 1$  for any  $x,y \in H\bigcap B_R$ and  $x\neq y.$
Indeed, we only need to prove  $\widetilde{a}b\geq a\widetilde{b}$.
Note that for any $x,y \in H\bigcap B_R$ and  $x\neq y$, we have
$(y,\sigma_Hx)<(y,x).$
Here  $(\cdot,\cdot)$ is the inner product of  $\mathbb{R}^N$.
Then,
\begin{displaymath}
\begin{array}{lll}
\widetilde{a}^2b^2-a^2\widetilde{b}^2&=&\frac{2|x|^2}{R^2}\{|y|^2(y,\widetilde{\sigma_Hx})+|y|^2(y,x)+|x|^2(y,\widetilde{\sigma_Hx})
+|\widetilde{\sigma_Hx}|^2(y,x)\}\\
 &&-\frac{2|x|^2}{R^2}\{|y|^2(y,\widetilde{x})+|y|^2(y,\sigma_Hx)+|\widetilde{x}|^2(y,{\sigma_Hx})+|\sigma_Hx|^2(y,\widetilde{x})\}\\
&=&\frac{2|x|^2}{R^2}\{\frac{R^2|y|^2}{|x|^2}(y,{\sigma_Hx})+|y|^2(y,x)+R^2(y,{\sigma_Hx})+\frac{R^4}{|x|^2}(y,x)\}\\
&&-\frac{2|x|^2}{R^2}\{\frac{R^2|y|^2}{|x|^2}(y,x)+|y|^2(y,\sigma_Hx)+\frac{R^4}{|x|^2}(y,{\sigma_Hx})+R^2(y,{x})\}\\
&=&\frac{2|x|^2}{R^2}(\frac{R^2}{|x|^2}-1)(R^2-|y|^2)[(y,x)-(y,\sigma_Hx)]\\
&\geq&0.
\end{array}
\end{displaymath}

\noindent\textbf{Step 2.} By \noindent\textbf{Step 1}, for any $x,y \in H\bigcap B_R$ and  $x\neq y$, we have
 $\frac{\widetilde{b}^{N-2}-\widetilde{a}^{N-2}}{b^{N-2}-a^{N-2}}\leq\frac{\widetilde{a}^{N-2}}{a^{N-2}}$ and  $\widetilde{b}^{N-2}\geq b^{N-2}.$
So
\begin{displaymath}
\begin{array}{lll}
G(x,y)-G(\sigma_Hx,y)&=&(\frac{1}{a^{N-2}}-\frac{1}{b^{N-2}})-(\frac{1}{\widetilde{a}^{N-2}}-\frac{1}{\widetilde{b}^{N-2}})\\
&=&(b^{N-2}-a^{N-2})(\frac{1}{(ab)^{N-2}}-\frac{1}{(\widetilde{a}\widetilde{b})^{N-2}}
\cdot\frac{\widetilde{b}^{N-2}-\widetilde{a}^{N-2}}{b^{N-2}-a^{N-2}})\\
&\geq&(b^{N-2}-a^{N-2})(\frac{1}{(ab)^{N-2}}-\frac{1}{(\widetilde{a}\widetilde{b})^{N-2}}\cdot
\frac{\widetilde{a}^{N-2}}{a^{N-2}})\\
&=&\frac{1}{a^{N-2}}(b^{N-2}-a^{N-2})(\frac{1}{b^{N-2}}-\frac{1}{\widetilde{b}^{N-2}})\\
&\geq&0.
\end{array}
\end{displaymath}
The proof is completed.
\end{proof}

Let $H$ be an open half-space in $\mathbb{R}^N$ with the origin $O\in \partial H$ and
 $u^H:B_R\to\mathbb{R}$ be  the polarization of $u\in H^1_0(B_R)$ defined  by
\begin{equation}
u^H(x)=
\left\{
\begin{array}{lll}
\max\{u(x),u(\sigma_Hx)\},\ \  x\in H\bigcap B_R,\\
\min\{u(x),u(\sigma_Hx)\},\ \  x\in B_R\backslash (H\bigcap B_R).
\end{array}\right.
\end{equation}
Let
$$A_u=\{x\in H\cap B_R:u(x)\geq u(\sigma_H x)\},\  B_u=\{x\in H\cap B_R:u(x)< u(\sigma_H x)\}.$$
Then we are ready to prove the separability property of  positive ground states of \eqref{17}.

\begin{prop}\label{prop4.1}
Suppose  $u$ is a  positive ground state  of \eqref{17}. Then
\begin{equation}\label{22}
\mathbb{D}(u^H)\geq\mathbb{D}(u).
\end{equation}
Moreover,  $u$ is separable in $B_R$.
\end{prop}
\begin{proof}
First for simplicity of notations, we write
$$a:=|u(x)|^p,\ b:=|u(\sigma_H x)|^p,\ c:=|u(y)|^p,d=|u(\sigma_H y)|^p.$$
It is easy to check
\begin{equation}\label{23}
\begin{array}{lll}
& &\mathbb{D}(u^H)-\mathbb{D}(u):=I_1+I_2+I_3+I_4,
\end{array}
\end{equation}
where
\begin{equation}\label{24}
\begin{array}{lll}
I_1&=&\displaystyle\int_{A_u}\int_{A_u}G(x,y)(ac-ac)+G(\sigma_H x,y)(bc-bc)dxdy\\
&&+\displaystyle\int_{A_u}\int_{A_u}G(x,\sigma_H y)(ad-ad)+G(\sigma_H x,\sigma_H y)(bd-bd)dxdy,
\end{array}
\end{equation}

\begin{equation}\label{25}
\begin{array}{lll}
I_2&=&\displaystyle\int_{A_u}\int_{B_u}G(x,y)(ad-ac)+G(\sigma_H x,y)(bd-bc)dxdy\\
&&+\displaystyle\int_{A_u}\int_{B_u}G(x,\sigma_H y)(ac-ad)+G(\sigma_H x,\sigma_H y)(bc-bd)dxdy,
\end{array}
\end{equation}

\begin{equation}\label{26}
\begin{array}{lll}
I_3&=&\displaystyle\int_{B_u}\int_{A_u}G(x,y)(bc-ac)+G(\sigma_H x,y)(ac-bc)dxdy\\
&&+\displaystyle\int_{B_u}\int_{A_u}G(x,\sigma_H y)(bd-ad)+G(\sigma_H x,\sigma_H y)(ad-bd)dxdy,
\end{array}
\end{equation}

\begin{equation}\label{27}
\begin{array}{lll}
I_4&=&\displaystyle\int_{B_u}\int_{B_u}G(x,y)(bd-ac)+G(\sigma_H x,y)(ad-bc)dxdy\\
&&+\displaystyle\int_{B_u}\int_{B_u}G(x,\sigma_H y)(bc-ad)+G(\sigma_H x,\sigma_H y)(ac-bd)dxdy.
\end{array}
\end{equation}
By Lemma~\ref{lem4.2}, we have  $I_1=I_4=0$ and
\begin{equation}\label{28}
\begin{array}{lll}I_2
=\displaystyle\int_{A_u}\int_{B_u}(G(x,y)-G(x,\sigma_H y))(a-b)(d-c)dxdy
\geq0
\end{array}
\end{equation}

\begin{equation}\label{29}
\begin{array}{lll}I_3
=\displaystyle\int_{B_u}\int_{A_u}(G(x,y)-G(x,\sigma_H y))(b-a)(c-d)dxdy\geq0.
\end{array}
\end{equation}
Hence \eqref{22} holds.

In addition,
\begin{displaymath}
\begin{array}{lll}
\displaystyle\int_{B_R}|\nabla u^H(x)|^2dx
&=&\displaystyle\int_{x\in H\bigcap B_R}|\nabla u^H(x)|^2dx+\int_{x\in H\bigcap B_R}|\nabla u^H(\sigma_H x)|^2dx\\
&=&\displaystyle\int_{A_u}|\nabla u(x)|^2dx+\int_{B_u}|\nabla u(\sigma_H x)|^2dx
+\displaystyle\int_{A_u}|\nabla u(\sigma_H x)|^2dx+\int_{B_u}|\nabla u(x)|^2dx\\
&=&\displaystyle\int_{x\in H\bigcap B_R}|\nabla u(x)|^2dx+\int_{x\in H\bigcap B_R}|\nabla u(\sigma_H x)|^2dx\\
&=&\displaystyle\int_{B_R}|\nabla u(x)|^2dx.
\end{array}
\end{displaymath}
Similarly,
$$\int_{B_R}|u^H(x)|^2dx=\int_{B_R}|u(x)|^2dx.$$
Since  $u$ is a ground state of \eqref{17}, by Lemma~\ref{lem4.1}, we deduce
$\mathbb{D}(u^H)\leq \mathbb{D}(u).$
This together with \eqref{22}, implies that
$\mathbb{D}(u^H)=\mathbb{D}(u).$ So $I_2=I_3=0,$ that is,
$A_u=\emptyset\ \mbox{or}\ B_u=\emptyset.$
Therefore, $u^H=u$ or $u^H=u\circ\sigma_H$, that is, \eqref{11} holds by the definition $u^H$.
\end{proof}

Now, Theorem~\ref{thm4.1} follows easily
from Proposition~\ref{prop4.1} and Theorem~\ref{thm2.1}.

\begin{thm}\label{thm4.1}
Let $N\geq3$ and $p\in(\frac{N+2}{N},\frac{N+2}{N-2}).$ Assume that  $u\in H^1_0(B_R)$ is a positive ground state of \eqref{17}. Then  either $u$ is radially symmetric with respect to the origin $O$ or there exists  $M\in O(N)$ such that  $u_M$  is only axially symmetric with respect to  $x_N$-axis and, for  $\alpha\in(0,R],$
$u_{M}(\alpha\cos\theta,0,\cdots,0,\alpha\sin\theta)$ is decreasing with respect to
$\theta\in[\frac{\pi}{2},\frac{3\pi}{2}]$.
\end{thm}

\subsection{Choquard type equations in whole space}
In this subsection, we consider the Choquard equation \eqref{30} in $\mathbb{R}^N$.  The qualitative properties  of ground states  of \eqref{30} have been intensively studied  in \cite{vv}. In particular, the separability of ground state is proved.
 \begin{lemma}\label{lem4.3}
Let $N\geq3, \alpha\in (0,N), p\in(\frac{N+\alpha}{N},\frac{N+\alpha}{N-2}).$ Assume  that  $u \in H^1(\mathbb{R}^N)$ is a positive ground state of  \eqref{30}. Then $u$ is separable in $\mathbb{R}^N.$
\end{lemma}
This lemma follows from the proof of Lemma 5.3 and Proposition 5.2 in \cite{vv}.
 Theorem~\ref{thm4.2}
is a direct consequence of Theorem~\ref{thm3.1} and Lemma \ref{lem4.3}.
 \begin{thm}\label{thm4.2}
Let $N\geq3, \alpha\in (0,N), p\in(\frac{N+\alpha}{N},\frac{N+\alpha}{N-2}).$ Assume  that  $u \in H^1(\mathbb{R}^N)$ is a positive ground state of  \eqref{30}. Then  there exist $x^*\in \mathbb{R}^N$ and a nonnegative decreasing function $v:(0,\infty)\to\mathbb{R}$ such that
$u(x)=v(|x-x^*|)$
and $\lim\limits_{r\to\infty}v(r)=0.$
\end{thm}

This paper  provides  a new and different perspective to study the symmetry and monotonicity of solutions to elliptic equations.
In the future work, on one hand, we want to investigate the symmetry and monotonicity of separable functions  in other symmetric domains rather than $B_R$ and $\mathbb{R}^N$. On the other hand, we are interested in the radial symmetry and uniqueness of ground states of Choquard type equations in $B_R$.  Furthermore, the relationship between the ground states of Choquard type equations in $B_R$ and  that in $\mathbb{R}^N$ when $R\to \infty$ is also worth studying.


\section*{Acknowledgement} Part of this work was done during Wang's visit of the Department of Mathematics, Sun Yat-sen University (Zhuhai). She appreciated the hospitality from the faculty and staff.

\bigskip

Tao Wang

College of Mathematics and Computing Science

Hunan University of Science and Technology

Xiangtan, Hunan 411201, P. R. China

Email: wt\_61003@163.com

\bigskip

Taishan Yi

School of Mathematics (Zhuhai)

Sun Yat-Sen University

Zhuhai, Guangdong 519082,  P. R. China

Email: yitaishan@mail.sysu.edu.cn

\begin{thebibliography}{99}
\bibitem{btw} 
Bartsch T, Weth T, Willem M. Partial symmetry of least energy nodal solution to some variational problems. J. Anal. Math. \textbf{96}, 1-18 (2005).

\bibitem{brock} 
Brock F, Solynin  A  Yu. An approach to symmetrization via polarization. Trans. Amer.
Math. Soc. \textbf{352}, 1759-1796 (2000).

\bibitem{CHENLI} 
Chen W X, Li C M. Classification of solutions of some nonlinear elliptic equations. Duke Math. J. \textbf{63}, 615-622 (1991).

\bibitem{CHENLI1} 
Chen W X, Li C M. Qualitative properties of solutions to some nonlinear elliptic equations in $\mathbb{R}^2$. Duke Math. J. \textbf{71}, 427-439 (1993).

\bibitem{chenli} 
Chen W X, Li C M. Methods on Nonlinear Elliptic Equations. American Institute of Mathematical Sciences, Springfield, 2010.

\bibitem{clo} 
Chen  W X, Li  C M, Ou  B.  Classification of solutions for an integral equation. Comm. Pure Appl. Math. \textbf{59}, 330-343 (2006).

\bibitem{clo1} 
Chen W X, Li C M, Li  Y. A direct method of moving planes for the fractional Laplacian. Adv. Math. \textbf{308}, 204-437 (2017).

\bibitem{Linchang} 
Chern J L, Lin C S. The symmetry of least-energy solutions for semilinear elliptic equations. J. Differ. Equ. \textbf{187}, 240-268 (2003).

\bibitem{cp} 
Choquard  P, Stubbe  J, Vuffracy  M.  Stationary solutions of the Schr\"{o}dinger-Newton model-An ODE approach. Differ. Integral  Equ.  \textbf{27}, 665-679 (2008).

\bibitem{els} 
Evans  L C.  Partial differential equations.  American Mathematical Society, Providence, RI, 2010.

\bibitem{fl} 
Frank R L, Lenzmann E, Silvestre L. Uniqueness of radial solutions for the fractional Laplacian. Comm. Pure Appl. Math. \textbf{69},  1671-1726 (2016).


\bibitem{gmv} 
Ghimenti M, Moroz V, Van Schaftingen J. Least action nodal solutions for the quadratic Choquard equation.  Proc. Amer. Math. Soc. \textbf{145}, 737-747 (2017).

\bibitem{gmvj}  
Ghimenti M, Van Schaftingen J. Nodal solutions for the Choquard equation. J. Funct. Anal.  \textbf{271}, 107-135 (2016).

\bibitem{GNN} 
Gidas B, Ni W M, Nirenberg L. Symmetry of positive solutions of nonlinear elliptic equations in $\mathbb{R}^N,$ in the book Mathematical Analysis and Applications, Academic Press, New York, 1981.

\bibitem{GNN1} 
Gidas B, Ni W M, Nirenberg L. Symmetry and related properties
via the maximum principle. Comm. Math. Phys. \textbf{68}, 209-243 (1979).

\bibitem{gmx} 
Gui C, Malchiodi A, Xu H. Axial symmetry of some steady state solutions to nonlinear Schr$\ddot{o}$dinger equations. Proc. Amer. Math. Soc.  \textbf{139}, 1023-1032 (2011).

\bibitem{kmk} 
Kwong M K. Uniqueness of positive solutions of $-\Delta u-u+u^p=0$ in $R^n$. Arch. Rat. Mech. Anal. \textbf{105}, 243-266 (1989).

\bibitem{Li2}   
Li C. Monotonicity and symmetry of solutions of fully nonlinear elliptic equations on bounded domains. Comm. Partial Differ. Equ. \textbf{16}, 491-526 (1991).

\bibitem{Li4} 
Li Y Y. Remark on some conformally invariant integral equations: the method of moving spheres. J.   Eur. Math. Soc. \textbf{6}, 153-180 (2004).

\bibitem{LEH} 
Lieb  E H. Existence and uniqueness of the minimizing solution of Choquard nonlinear equation. Stud. Appl. Math. \textbf{57}, 93-105 (1977).

\bibitem{llm} 
Lieb E H, Loss M. Analysis, Graduate studies in mathematics. American Mathematical Society, Providence, RI, 1997.

\bibitem{LPL} 
Lions  P L.  
The Choquard equation and related questions. Nonlinear Anal. \textbf{4}, 1063-1073 (1980).

\bibitem{MZ} 
Ma  L, Zhao  L. Classification of positive solitary solutions of the nonlinear Choquard equation. Arch. Rat. Mech. Anal. \textbf{195}, 455-467 (2010).

\bibitem{vv} 
Moroz  V, Van Schaftingen  J. Groundstates of nonlinear Choquard equations: existence, qualitative properties and decay estimates. J. Funct. Anal.  \textbf{265}, 153-184 (2014).

\bibitem{rudin} 
Rudin W. Functional analysis. In: International series in pure and applied mathematics.
second ed. McGraw-Hill, New York, 1991.

\bibitem{SZOU} 
Serrin J, Zou H. Symmetry of ground states of quasilinear elliptic equations. Arch. Rat.  Mech. Anal. \textbf{148}, 265-290 (1999).

\bibitem{tm}  
Tod  P, Moroz  M  I. An analytical approach to the Schr\"{o}dinger-Newton equations. Nonlinearity  \textbf{12}, 201-216 (1999).

\bibitem{schaftingen} 
Van Schaftingen J. Explicit approximation of the symmetric rearrangement by polarizations. Arch. Math. (Basel) \textbf{93}, 181-190 (2009).

\bibitem{sjvm}  
Van Schaftingen J, Willem M. Symmetry of solutions of semilinear elliptic problems. J. Eur. Math. Soc. \textbf{10}, 439-456 (2008).

\bibitem{wtyt} 
Wang T, Yi T. Uniqueness of positive solutions of the
 Choquard type equations. Appl. Anal. \textbf{96}, 409-417 (2017).
\end{thebibliography}
\end{document}